\documentclass[12pt]{article}
\usepackage{geometry}
\usepackage{float}
\usepackage{latexsym,bm}
\usepackage{amsmath,amsfonts,amsthm,amssymb,mathrsfs}
\usepackage{graphicx}
\usepackage{hyperref}
\hypersetup{hidelinks}
\usepackage{xcolor}
\usepackage{mathtools}
\usepackage{enumitem}
\usepackage{comment}

\newtheorem{thm}{Theorem}[section]

\newtheorem{lemma}{Lemma}[section]
\newtheorem{cor}{Corollary}[section]
\newtheorem{pro}{Proposition}[section]

\newtheorem{definition}{Definition}[section]

\newcommand{\E}{\mathbb{E}}
\renewcommand{\Pr}{\mathbb{P}}

\title{The induced-$P_4$-free process}
\author{Hongyi Lou\thanks{Academy of Mathematics and Systems Science, Chinese Academy of Sciences, Beijing, China, and University of Chinese Academy of Sciences, Beijing, China.}
\and Xinzhe Song\footnotemark[1]
\and Guiying Yan\footnotemark[1]
}
\date{}

\begin{document}

\maketitle

\begin{abstract}

We study the random induced-\(P_4\)-free graph process. Let \(e_1,\ldots,e_N\), where \(N=\binom{n}{2}\), be a uniformly random
ordering of the edges of \(K_n\). Starting from the empty graph \(G_0\), we add \(e_{m+1}\) whenever \(G_m+e_{m+1}\) contains no induced \(P_4\),
and otherwise leave the graph unchanged. We show that the terminal graph is a trivially perfect graph and we describe the structure and distribution of the connected components of the terminal graph $G_N$. Consequently, we derive the limiting values of several natural graph parameters. In particular, the terminal graph $G_N$ has \(\Theta(n)\) edges.


\end{abstract}

\section{Introduction}\label{sec:introduction}

The modern study of random graph processes originated with the seminal work of Erd\H{o}s and R\'enyi \cite{ErdosRenyi1959,ErdosRenyi1960}. Let \(e_1,\ldots,e_N\), where \(N=\binom{n}{2}\), be a uniformly random permutation of the edges of \(K_n\). They considered the properties of $G_{n,m}$ under a uniformly random ordering of the edges $\{e_1,e_2,\cdots, e_m\}$, which is known as the Erd\H{o}s-R\'enyi random graph.

A variant of the Erd\H{o}s-R\'enyi random graph is the random graph process, in which the edges of $K_n$ are offered in a uniformly random order and an offered edge is accepted if and only if the graph continues to satisfy a prescribed graph property, for example, triangle-free process, cycle-free process and  general $H$-free process \cite{ErdosSuenWinkler1995,Bohman2009,BohmanKeevash2010,Warnke2011,WarnkeK42014,WarnkeCl2014}.  These graph properties are  monotone decreasing and they are analyzed by the differential equation method of Wormald \cite{Wormald1995,Wormald1999}. These properties can improve lower bounds on Tur\'an numbers and on off-diagonal Ramsey numbers. Random graph processes governed by non-monotone graph properties are much less understood. The notable example is the K\H{o}nig process\cite{KamcevKMS2020}, which keeps the maximum matching number equal to the optimal vertex cover by offering available edges. Another natural source of non-monotone constraints arises from forbidding induced subgraphs.

In this paper, we consider the non-monotone graph property without induced \(P_4\), where \(P_4\) denotes the path on four vertices and three edges. A graph containing no induced \(P_4\) is also called a \emph{cograph} \cite{CorneilLerchsBurlingham1981}. Uniformly sampled random cographs have also been studied for the graph limits and degree distributions\cite{BassinoEtAl2020,Stufler2021}. Unlike the \(H\)-free graph properties we mentioned above, the induced-\(P_4\)-free property is non-monotone with respect to the edge set: deleting an edge may create an induced \(P_4\), while adding a further edge may also create an induced $P_4$. 

In this paper, we define the induced-\(P_4\)-free process as follows. Let \(G_0\) be the empty graph on a vertex set \(V\) with \(|V|=n\), and let \(N=\binom{n}{2}\). Let $\{e_1,e_2,\ldots,e_N\}$ be a uniformly random ordering of the edges of the complete graph \(K_n\) on \(V\). At each step \(m\geq 0\), the edge \(e_{m+1}\) is given to \(G_m\). We say $e_{m+1}$ is \emph{acceptable} for \(G_m\) if \(G_m+e_{m+1}\) contains no induced copy of \(P_4\). If \(e_{m+1}\) is acceptable, we set \(G_{m+1}=G_m+e_{m+1}\) and say that it is accepted; otherwise, we set \(G_{m+1}=G_m\) and say that it is rejected. We write \(G_N\) for the terminal graph. An important observation is that the process also preserves the induced-\(C_4\)-free property, so every \(G_m\) is an induced-$\{ P_4,C_4 \}$-free graph , which is also called a trivially perfect graph\cite{Golumbic1978}. 

The principal difficulty is that the set of acceptable edges is not monotone during the process, as discussed above, an edge that is unacceptable at one step may become acceptable later. Consequently, the standard differential equation approach used for \(H\)-free processes cannot be applied directly here. To avoid this difficulty, rather than considering the entire process, we focus on a monotone auxiliary graph process, called the I-skeleton process, which uses the same random edge ordering but ignores edges between two non-isolated vertices. Consequently, every non-trivial component of the auxiliary process is a star, and its relevant variables can be analyzed by the differential equation method. We track this process until only $o(n)$ isolated vertices remain. We then show that the original process is close to the auxiliary process. Our main goal is to describe the component structure of the terminal graph $G_N$ and, as consequences of this structural description, derive limiting results for several natural graph parameters.

Our main result is that we give an asymptotic description of the component decomposition of the terminal graph $G_N$. Since we can describe the components of the terminal graph $G_N$, we can get many parameters for the graph. For example, we obtain the number of different components that are isomorphic to a fixed connected graph, we can also count the number of cliques $K_t$,  we also know the degree distribution and so on. In particular, the terminal graph $G_N$ has \(\Theta(n)\) edges with high probability.

The paper is organized as follows. The main results are presented in Section~\ref{sec:main-results}. We introduce the tools we used in Section~\ref{sec:preliminaries}.  We analyze the auxiliary graph process by the differential equation method in Section~\ref{sec:iskeleton-process}. We then compare it to the original process and show that they are close in Section~\ref{sec:iskeleton-transfer}. Finally, we consider the graph process from the point when $o(n)$ isolated vertices remain to the final graph, and we prove all the theorems in Section~\ref{sec:terminal-phase}.

\section{Main results}\label{sec:main-results}
In this section, we state the main results of the article. 

Let \(\mathcal C(G)\) denote the set of connected components of \(G\). For \(C\in\mathcal C(G)\), we identify \(C\) with its vertex set whenever no confusion can arise, and write \(|C|:=|V(C)|\).

For graphs \(A\) and \(B\) on disjoint vertex sets, let \(A\vee B\) denote their join, obtained from their disjoint union by adding all edges between \(V(A)\) and \(V(B)\). For each \(\ell\ge1\), let \(\mathcal G_\ell\) denote the terminal random graph of the induced-\(P_4\)-free process on the labelled vertex set \([\ell]\), and define \(\mathcal H_\ell:=K_1\vee\mathcal G_\ell\). Thus, \(\mathcal H_\ell\) is the random graph obtained from a star with exactly \(\ell\) leaves by running an independent induced-\(P_4\)-free process on its leaf set. Define
\[
\rho_1=\frac{\sqrt2(2-\sqrt2)}{4}e^{-\sqrt2},
\qquad
\rho_\ell=e^{-\sqrt2}
\left(
\frac{(\sqrt2)^\ell}{\ell!}
-\frac{(\sqrt2)^{\ell+1}}{(\ell+1)!}
\right)
\quad(\ell\ge2).
\]

The following theorem gives an asymptotic decomposition of the terminal graph.

\begin{thm}
\label{thm:intro-terminal-decomposition}
For every fixed integer \(\ell\ge1\), the terminal graph \(G_N\) contains \( \rho_\ell n+o_{\mathbb P}(n) \) components obtained by starting from a star with $\ell $ leaves and subsequently running an induced-$P_4$-free process on its leaf set. Each such random component has distribution \(\mathcal H_\ell\).

Let \(\mathcal E_n\) denote the family of all remaining components of $G_N$. For every fixed \(q\ge1\),
\[
\sum_{C\in\mathcal E_n}|V(C)|^q=o_{\mathbb P}(n).
\]
In particular, the exceptional components contain \(o_{\mathbb P}(n)\) vertices in total.
\end{thm}

The constants \(\rho_\ell\) are determined by the evolution of the process. Let \(\widehat I_r\), \(\widehat C_r\), and \(\widehat Y_{1,r}\) be the numbers of isolated vertices, non-trivial components and \(K_2\)-components after offering $r$ edges, respectively. 
We rescale the discrete offer time \(r\) by setting \(s_r:=\frac{2r}{n}\).
Equivalently, for \(s\ge0\), let \( r(s):=\left\lfloor sn/2\right\rfloor
\)
and define
\[
\widehat i(s)=\frac{\widehat I_{r(s)}}{n},
\qquad
\widehat c(s)=\frac{\widehat C_{r(s)}}{n},
\qquad
\widehat y_1(s)=\frac{\widehat Y_{1,r(s)}}{n}.
\]

\begin{thm}
\label{thm:intro-iskeleton}
For every fixed finite time horizon \(S<\infty\),
\[
\sup_{0\le s\le S}
\left\|
(\widehat i(s),\widehat c(s),\widehat y_1(s))-(i(s),c(s),y_1(s))
\right\|_\infty
\xrightarrow{\mathbb P}0,
\]
where $(i,c,y_1)$ is the solution to the following system of differential equations:
\[
        \dot i=-i^2-i(c+y_1),\qquad
        \dot c=\frac12 i^2,\qquad
        \dot y_1=\frac12 i^2-2iy_1,
\]
with initial conditions \((i(0),c(0),y_1(0))=(1,0,0)\). 

Moreover, for every fixed \(\ell\ge1\), if \(\widehat Y_{\ell,N}\) denotes the terminal number of stars with exactly \(\ell\) leaves, then we have
\[
        \frac{\widehat Y_{\ell,N}}{n}
        \xrightarrow{\mathbb P}\rho_\ell
\]

\end{thm}

For a finite connected graph \(F\), let \(C_F(G)\) be the number of
components of \(G\) isomorphic to \(F\). We obtain the following corollary:

\begin{cor}
\label{cor:fixed-component-limits}
For every fixed finite connected graph \(F\) with \(|V(F)|\ge 2\),
\[
\frac{C_F(G_N)}{n}
\xrightarrow{\mathbb P}
\rho_{|V(F)|-1}
\mathbb P\!\left(
\mathcal H_{|V(F)|-1}\cong F
\right).
\]

\end{cor}

Since the terminal graph can be viewed as the union of many components described in Theorem~\ref{thm:intro-terminal-decomposition},  we can now consider some graph parameters. Let \(\phi(D)\) be a real-valued graph parameter on finite graphs satisfying \(|\phi(D)|\le q_1|V(D)|^{q_2}\) for fixed \(q_1,q_2<\infty\), for example, the number of unlabelled copies of a clique, the distribution of vertex degree and so on. Then we have the following theorem and corollary:

\begin{thm}
\label{thm:additive-statistics}
For every graph parameter $\phi$ satisfying the above condition
\[
        \frac1n\sum_{D\in\mathcal C(G_N)}\phi(D)
        \xrightarrow{\mathbb P}
        \sum_{\ell\ge1}\rho_\ell\,\mathbb E\phi(\mathcal H_\ell),
\]
and the series on the right is absolutely convergent.
\end{thm}

For \(\ell\ge1\), write
\(
        a_\ell:=\mathbb E\,e(\mathcal G_\ell)
\) and denote the number of unlabelled copies of fixed graph $F$ in graph $H$ by $N(H,F)$.

\begin{cor}
\label{cor:intro-clique-profile}
For every fixed \(j\ge2\),
\[
        \frac{N(G_N,K_j)}n
        \xrightarrow{\mathbb P}
        \kappa_j:=
        \sum_{\ell\ge1}\rho_\ell\,
        \mathbb E N(\mathcal H_\ell,K_j).
\]
In particular, if $j=2$, then
\[
        \frac{e(G_N)}{n}\xrightarrow{\mathbb P}\kappa_2 =1-\frac{1+\sqrt2}{2}e^{-\sqrt2}
        +\sum_{\ell\ge1}\rho_\ell a_\ell>0.
\]
As $\infty>\kappa_2$ is a constant, \(G_N\) has \(\Theta(n)\) edges w.h.p.
\end{cor}

We can also determine the distribution of vertex degree. For a finite graph \(H\), let \(N_k^{\deg}(H)=|\{v\in V(H):d_H(v)=k\}|\).

\begin{cor}
\label{cor:intro-degree-distribution}
For every fixed \(k\ge0\),
\[
        \frac{N_k^{\deg}(G_N)}{n}
        \xrightarrow{\mathbb P}p_k,
        \qquad
        p_k=\sum_{\ell\ge1}\rho_\ell\,
        \mathbb E N_k^{\deg}(\mathcal H_\ell).
\]
Moreover, \(p_0=0\) and \(\sum_{k\ge0}p_k=1\).
\end{cor}

The proof proceeds in three steps.  We first analyze the auxiliary I-skeleton process by the differential equation method in Section~\ref{sec:iskeleton-process}. We then couple it to the original process and show that the exceptional components can be controlled in Section~\ref{sec:iskeleton-transfer}.  Finally, the clean stars are completed independently inside their leaf sets, which gives Theorem~\ref{thm:intro-terminal-decomposition} and the stated consequences in Section~\ref{sec:terminal-phase}.

\section{Preliminaries}\label{sec:preliminaries}

Let $\mathcal F_0\subseteq \mathcal F_1\subseteq \cdots\subseteq \mathcal F_N$ 
be the natural filtration generated by the random edge ordering.
Equivalently, \(\mathcal F_r\) records the history of the process up to the $r$-th
step. All stopping times considered below are with respect to this
filtration or to the corresponding stopped filtrations.

We introduce the following diagonal deterministic-envelope lemma.

\begin{lemma}\label{lem:deterministic-envelope}
Let \(b_n>0\) satisfy \(b_n\to\infty\), and \(W_n\ge0\). If for every fixed integer $m\ge1$, $b_n^mW_n\xrightarrow{\mathbb P}0$, then there is a sequence $\varepsilon_n>0$ with $\varepsilon_n\to0$ such that $W_n/\varepsilon_n\xrightarrow{\mathbb P}0$ and for every fixed integer $m\ge1$, $\varepsilon_nb_n^m\to0$.
\end{lemma}

\begin{proof}
We choose a slowly increasing sequence $(j(n))_{n\ge 1} \subseteq \mathbb Z_+$. Choose integers $1=n_1^\star<n_2^\star<\cdots$ so that, for every $i\ge1$ and all
$n\ge n_i^\star$,
$
        \mathbb P\!\left(b_n^{2i}W_n>i^{-2}\right)\le i^{-1}.
$
Let $j(n)=i$ for $n_i^\star\le n<n_{i+1}^\star$, and set
$\varepsilon_n=1/{j(n)b_n^{j(n)}}$.

On the event $b_n^{2j}W_n\le j^{-2}$, with $j=j(n)$,
$
        W_n/\varepsilon_n
        =jb_n^jW_n
        \le j^{-1}b_n^{-j}
        \le j^{-1}.
$
The complementary event has probability at most $j^{-1}$, and hence
$W_n/\varepsilon_n\to0$ in probability. Since \(b_n\to\infty\), we may ignore finitely many \(n\) and assume \(b_n\ge1\). For a fixed $m$, we also may ignore finitely many \(n\) and assume \(j(n)>m\) and then
$
        \varepsilon_nb_n^m
        =j(n)^{-1}b_n^{m-j(n)}
        \le j(n)^{-1}\to0.
$
\end{proof}

We shall use the following form of Freedman's inequality
\cite{Freedman1975}; see also \cite{ChungLu2006}.  It packages the localization
step used throughout the proof.

\begin{lemma}\label{lem:freedman}(Freedman's inequality)
Let $(\mathcal F_t)_{t\ge0}$ be a filtration and $\tau$ be a stopping time, and let $(M_t)_{t\ge0}$ be an $(\mathcal F_t)$-adapted real-valued process
with $M_0=0$. Here \(t\wedge\tau:=\min\{t,\tau\}\). Suppose that the stopped process $(M_{t\wedge\tau})_{t\ge0}$ is a supermartingale and for some $C>0$,
$
        M_{t\wedge\tau}-M_{(t-1)\wedge\tau}\le C
$
almost surely for every $t\ge1$.  Define
$
V_t=
\sum_{s=1}^t
\mathbb E\!\left[
\bigl(M_{s\wedge\tau}-M_{(s-1)\wedge\tau}\bigr)^2
\middle|\mathcal F_{s-1}
\right].
$
Then for all $a,v>0$,
$$
\mathbb P\!\left(
\exists t\ge0:
M_{t\wedge\tau}\ge a
\ \text{and}\
V_t\le v
\right)
\le
\exp\!\left(
-\frac{a^2}{2(v+Ca/3)}
\right).
$$
In particular, if the stopped process is a martingale and $\left|M_{t\wedge\tau}-M_{(t-1)\wedge\tau}\right|\le C$ almost surely for every $t\ge1$, then
$$
\mathbb P\!\left(
\exists t\ge0:
|M_{t\wedge\tau}|\ge a
\ \text{and}\
V_t\le v
\right)
\le
2\exp\!\left(
-\frac{a^2}{2(v+Ca/3)}
\right).
$$
\end{lemma}

We also introduce the following inequality for subsequent use.

\begin{lemma}\label{lem:discrete-gronwall}
Let $(u_i)_{i\ge0}$, $(a_i)_{i\ge0}$, and $(b_i)_{i\ge0}$ be
non-negative real sequences.  Suppose that for every $0\le i<m$,
$        u_{i+1}\le (1+a_i)u_i+b_i. $
Then
$$
        u_m\le u_0\prod_{i=0}^{m-1}(1+a_i)
        +\sum_{j=0}^{m-1} b_j\prod_{i=j+1}^{m-1}(1+a_i).
$$
In particular,
$$
        u_m\le \exp\!\left(\sum_{i=0}^{m-1}a_i\right)
        \left(u_0+
        \sum_{j=0}^{m-1} b_j\right).
$$

\end{lemma}

\begin{proof}
Iterating the recursion gives
$$
        u_m\le u_0\prod_{i=0}^{m-1}(1+a_i)
        +\sum_{j=0}^{m-1} b_j\prod_{i=j+1}^{m-1}(1+a_i).
$$
The exponential bound follows from $1+x\le e^x$ for $x\ge0$.
\end{proof}

We shall use the following lemma for a uniformly random ordering $\pi$.

\begin{lemma}\label{lem:suffix-exchangeability}
After revealing the ordering up to any stopping time \(T\), let \(E_T\) be the set of unrevealed elements. For $A \subseteq E_T$, denote the first position at which an element of $A$ appears by $\min_\pi A$.  Conditional on \(\mathcal F_T\), the unrevealed suffix is still a uniform random ordering of \(E_T\).  In particular, for disjoint unrevealed nonempty sets \(A,B\), 
\[
        \mathbb P(\min_\pi A<\min_\pi B\mid\mathcal F_T)
        =
        \frac{|A|}{|A|+|B|} \le \frac{|A|}{|B|},
\]
Moreover, if \(A_1,\ldots,A_m\subseteq E_T\) are pairwise disjoint nonempty
\(\mathcal F_T\)-measurable sets, then the relative orders induced by the suffix
on \(A_1,\ldots,A_m\) are mutually independent uniform orders.
\end{lemma}
\begin{proof}
This follows immediately from the exchangeability of the unrevealed elements
of a uniform random ordering after conditioning on the revealed prefix.  The
race identity and the independence statement follow by symmetry.
\end{proof}

\section{The I-skeleton decomposition and evolution}\label{sec:iskeleton-process}
This section has two purposes. First, we prove that the graph during the evolution must be trivially perfect. Second, we define and analyze the auxiliary process using differential equations. 

\subsection{Structure during the evolution}

For a connected graph \(C\), let
\[
        U(C)=\{x\in V(C): xy\in E(C)\text{ for every }y\in V(C)\setminus\{x\}\}
\]
be the set of \emph{universal vertices} of \(C\). A graph is
\emph{trivially perfect} if every non-empty connected induced subgraph has a
universal vertex\cite{Wolk1965,Golumbic1978}.

We give some facts during the whole evolution as follows:

\begin{pro}\label{pro:structural-invariants}
The following facts hold throughout the process.
\begin{enumerate}[label=(\roman*)]
\item No accepted edge joins two distinct non-trivial components.
\item Every graph appearing in the process is trivially perfect.
\item Let $C$ be a non-trivial component, $v\notin V(C)$ be an isolated vertex, and  $w\in V(C)$.  Then $vw$ is acceptable if and only if
$w\in U(C)$.
\end{enumerate}
\end{pro}

\begin{proof}
For (i), let $A$ and $B$ be distinct non-trivial components, and suppose that
$u\in A$ and $v\in B$.  Choose neighbors $u'\in A$ of $u$ and $v'\in B$ of
$v$.  After adding $uv$, the vertices $u',u,v,v'$ induce the path
$u'-u-v-v'$, so the edge is rejected.

For (ii), every graph produced by the process is induced-$P_4$-free by definition.
We show inductively that it is also induced-$C_4$-free.  If the addition of an accepted
edge $e$ created an induced $C_4$, then that cycle would contain $e$, since the
preceding graph was induced-$C_4$-free.  Deleting $e$ from the cycle would leave an
induced-$P_4$ in the preceding graph, which is a contradiction. Thus every graph
appearing in the process is induced-$\{P_4,C_4\}$-free, and hence trivially perfect.

For (iii), suppose first that $w\notin U(C)$.  Choose $u\in V(C)$ non-adjacent to $w$, and choose a universal vertex $c\in U(C)$, which exists by (ii).  After adding $vw$, the vertices $v,w,c,u$ induce the path $v-w-c-u$, so the edge is rejected.  Conversely, suppose that $w\in U(C)$.  Any induced $P_4$ created by adding $vw$ would have to contain the new edge and, since $v$ has no other neighbor, would have the form $v$-$w$-$x$-$y$.  But $wy$ is then a chord because $w$ is universal in $C$. Hence no induced $P_4$ is created, and $vw$ is acceptable.
\end{proof}

The proposition describes the global evolution of components in two ways: creating a new non-trivial component from two isolated vertices, and attaching an isolated vertex to a currently universal vertex.

\subsection{Auxiliary I-skeleton process}
\label{subsec:coupled-skeleton}

The auxiliary I-skeleton process is coupled to the original induced-$P_4$-free
process by using the same uniformly random ordering
$(e_1,e_2,\ldots,e_N)$ of $E(K_n)$.  Let $\widehat G_0$ be the empty
graph on $V$.  Suppose that $\widehat G_r$ has been constructed and that
the next offered edge is $e_{r+1}=xy$.

The auxiliary process is defined recursively as follows.
Set
\[
        \widehat G_{r+1}=
        \begin{cases}
        \widehat G_r+xy, & \widehat G_r+xy \text{ contains no induced copy of $P_4$ and } \{x,y\}\cap I(\widehat G_r)\neq \emptyset,\\
        \widehat G_r, & \text{otherwise}.
        \end{cases}
\]

By induction, every non-trivial component of $\widehat G_r$ is a star.
For a copy of $K_2$, both vertices belong to $U(C)$, so either vertex
may serve as its center.

Relative to a graph $H$, we call $xy$ an \emph{I--I edge} if both
endpoints are isolated in $H$, an \emph{I--N edge} if exactly one
endpoint is isolated in $H$, and an \emph{N--N edge} if neither endpoint
is isolated in $H$. 

For \(ab\in\{\mathrm{II},\mathrm{IN}\}\), let \(\widehat X_{ab}(r)\)
denote the number of edges of type \(ab\) that have been added to
\(\widehat G\) among \(e_1,\ldots,e_r\).  Here \(\mathrm{II}\) and
\(\mathrm{IN}\) are the subscript forms of the I--I and I--N edge types,
respectively.

For $r\ge0$, recall that \(\widehat I_r\) and \(\widehat C_r\) denote
the numbers of isolated vertices and non-trivial components after
offering $r$ edges, respectively.  For $\ell\ge1$, let
\(\widehat Y_{\ell,r}\) denote the number of stars with exactly $\ell$
leaves after offering $r$ edges.  Their initial values are
\[
        \widehat I_0=n,\qquad
        \widehat C_0=0,\qquad
        \widehat Y_{\ell,0}=0.
\]

A copy of $K_2$ is counted as a one-leaf star. We use the normalized variables
$$
        \widehat i_r=\frac{\widehat I_r}{n},\qquad
        \widehat c_r=\frac{\widehat C_r}{n},\qquad
        \widehat y_{\ell,r}=\frac{\widehat Y_{\ell,r}}{n}\quad(\ell\ge1).
$$
For fixed $L$, we write $\widehat{\mathbf x}_r=(\widehat i_r,\widehat c_r,\widehat y_{1,r}),
        \
        \widehat{\mathbf x}_r^{(L)}
        =(\widehat i_r,\widehat c_r,\widehat y_{1,r},\ldots,\widehat y_{L,r}).$

 We use the continuous time $s_r:=2r/n$,
so one step corresponds to the increment $\Delta s=2/n$.  The number of
unoffered pairs at time $r$ is $M_r=\binom n2-r$.
Fix $S<\infty$ and restrict $0\le r\le Sn/2$. Then $M_r=\frac{n^2}{2}\bigl(1+O(n^{-1})\bigr)$.

\begin{pro}\label{pro:closed-iskeleton-evolution}

For fixed $L$, let $\mathbf x^{(L)}(s)=(i(s),c(s),y_1(s),\ldots,y_L(s))$ be the solution of
$$
\dot i=-i^2-i(c+y_1),\qquad
\dot c=\frac12 i^2,\qquad
\dot y_1=\frac12 i^2-2iy_1,
$$
$$
\dot y_2=2iy_1-iy_2,
\qquad
\dot y_\ell=i(y_{\ell-1}-y_\ell),\quad 3\le\ell\le L,
$$
with initial condition $\mathbf x^{(L)}(0)=(1,0,0,0,\ldots,0)$.  Then, for every fixed
$S,L<\infty$,
$$
\sup_{0\le r\le Sn/2}
\left\|\widehat{\mathbf x}_r^{(L)}-\mathbf x^{(L)}(s_r)\right\|
\xrightarrow{\mathbb P}0.
$$

\end{pro}

\begin{proof}
All pairs counted in the transition estimates have not yet been offered, otherwise, it would have been accepted at its first offer. Thus $\binom{\widehat I_r}{2}$ gives the I-I estimate.  A
star with at least two leaves has one universal vertex, while $K_2$ has two;
therefore the number of acceptable I-N pairs is
$\widehat I_r(\widehat C_r+\widehat Y_{1,r})$, the number of pairs from isolated vertices to
$K_2$'s is $2\widehat I_r\widehat Y_{1,r}$, and the number of pairs from isolated vertices
to the centres of $\ell$-leaf stars is $\widehat I_r\widehat Y_{\ell,r}$.  Dividing by
$M_r=\binom n2-r$ gives the displayed probabilities as follows:

$$
\Pr(\text{next edge is I-I}\mid\mathcal F_r)=\widehat i_r^2+O(n^{-1}),
$$
$$
\Pr(\text{next edge is acceptable I-N}\mid\mathcal F_r)
=2\widehat i_r(\widehat c_r+\widehat y_{1,r})+O(n^{-1}),
$$
$$
\Pr(\text{next edge attaches to a }K_2\mid\mathcal F_r)=4\widehat i_r\widehat y_{1,r}+O(n^{-1}),
$$
$$
\Pr(\text{next edge attaches to an $\ell$-leaf star}\mid\mathcal F_r)
=2\widehat i_r\widehat y_{\ell,r}+O(n^{-1}).
$$

An I-I edge changes
$(\widehat I,\widehat C,\widehat Y_1)$ by $(-2,+1,+1)$; an I-N edge into
$K_2$ changes $(\widehat I,\widehat Y_1,\widehat Y_2)$ by
$(-1,-1,+1)$; and an I-N edge into an $\ell$-leaf star, $\ell\ge2$,
changes $(\widehat I,\widehat Y_\ell,\widehat Y_{\ell+1})$ by $(-1,-1,+1)$.  Combining
these increments with the probabilities gives the following drift estimates:

$$
\mathbb E[\Delta \widehat I_r\mid \mathcal F_r]
=-2\widehat i_r^2-2\widehat i_r(\widehat c_r+\widehat y_{1,r})+O(n^{-1}),
$$
$$
\mathbb E[\Delta \widehat C_r\mid \mathcal F_r]=\widehat i_r^2+O(n^{-1}),
$$
$$
\mathbb E[\Delta \widehat Y_{1,r}\mid\mathcal F_r]=\widehat i_r^2-4\widehat i_r\widehat y_{1,r}+O(n^{-1}),
$$
$$
\mathbb E[\Delta \widehat Y_{2,r}\mid\mathcal F_r]=4\widehat i_r\widehat y_{1,r}-2\widehat i_r\widehat y_{2,r}+O(n^{-1}),
$$
$$
\mathbb E[\Delta \widehat Y_{\ell,r}\mid\mathcal F_r]
=2\widehat i_r(\widehat y_{\ell-1,r}-\widehat y_{\ell,r})+O(n^{-1}),\qquad 3\le \ell\le L.
$$

Since we use the normalized variables and continuous time by $\Delta s=2/n$, dividing their one-step conditional expectations by $2/n$ gives exactly the differential equations. The conditions for applying the differential equation method in the form of Warnke (~\cite{Warnke2019}, Theorem 2) are satisfied, as the following properties hold: each unnormalized coordinate changes by at most $2$, the trend error is $O(n^{-1})$, the drift map is polynomial and hence Lipschitz on fixed bounded boxes, and the initial normalized vector is exactly $(1,0,0,0,\ldots,0)$.  Applying the differential equation method, we get the desired result.
\end{proof}

We next solve the differential method for \((i,c,y_1)\) and the solution helps to obtain the component density weights \((\rho_\ell)_{\ell\ge1}\).

\begin{pro}\label{pro:iskeleton-solutions}
Let
$
        u(s)=\int_0^s i(\tau)\,d\tau.
$
Then, as functions of $u$,
$$
i(u)=\frac{2-u^2}{2}e^{-u},
\qquad
c(u)=\frac{u(u+2)}{4}e^{-u},
\qquad
y_1(u)=\frac{u(2-u)}{4}e^{-u}.
$$
Moreover, \(u(s)\uparrow\sqrt2\) as \(s\to\infty\).
\end{pro}

\begin{proof}
While $i(s)>0$, the change of variable $du/ds=i(s)$ gives
$$
\frac{di}{du}=-i-c-y_1,
\qquad
\frac{dc}{du}=\frac12 i,
\qquad
\frac{dy_1}{du}=\frac12 i-2y_1.
$$
A direct substitution verifies the formulae and the initial
conditions.  The first zero of $i(u)$ is at $u=\sqrt2$, and $i(u)>0$ for
$0\le u<\sqrt2$.  Hence, the supremum of $u(s)$ is $\sqrt{2}$.
\end{proof}
To extend the compact-time convergence to terminal quantities, fix \(\eta>0\) and choose \(S\) such that \(i(S)<\eta\) and \(\sqrt2-u(S)<\eta\). Then, w.h.p., \(\widehat I_{\lfloor Sn/2\rfloor}\le 2\eta n\). Every subsequent I-skeleton edge consumes at least one isolated vertex, hence, at most \(2\eta n\) further I-skeleton edges are accepted. Consequently, the terminal values of \(\widehat C/n\) and of each fixed \(\widehat Y_\ell/n\) differ from their values at time \(\lfloor Sn/2\rfloor\) by \(O(\eta)\). Letting first \(n\to\infty\) and then $\eta \to 0$ yields the terminal limits obtained by the solution at \(u=\sqrt2\).

It follows that the limiting density of I-skeleton components, equivalently the
limiting density of I-I edges, is
$\beta=c(\sqrt2)=\frac{1+\sqrt2}{2}e^{-\sqrt2}.$
Each I-I edge creates exactly one I-skeleton component, and other I-skeleton edges do not change the number of I-skeleton components. Moreover, once all but $o_{\mathbb P}(n)$ isolated vertices have been consumed, we have $2\widehat X_{\mathrm{II}}(N)+\widehat X_{\mathrm{IN}}(N)=n+o_{\mathbb P}(n)$.
For the proof of Theorem~\ref{thm:intro-iskeleton}, we only need the finer distribution of star sizes.

\begin{proof}[proof of Theorem~\ref{thm:intro-iskeleton}]
Propositions~\ref{pro:closed-iskeleton-evolution} and \ref{pro:iskeleton-solutions} give the convergence, the differential equations for $(i,c,y_1)$, the explicit solutions, and the fact that the supremum of $u(s)$ is $\sqrt{2}$. It remains only to prove the terminal limits for the star $\widehat Y_\ell$.

Fix $L$. The joint compact-time convergence of the normalized process
\[
(\widehat i_r,\widehat c_r,
\widehat y_{1,r},\ldots,\widehat y_{L,r})
\]
follows from Proposition~\ref{pro:closed-iskeleton-evolution}. Changing variables to
$u=\int_0^s i(\tau)\,d\tau$ gives
$$
y_1'=\frac12 i-2y_1,
\qquad
y_2'=2y_1-y_2,
\qquad
y_\ell'=y_{\ell-1}-y_\ell\quad(\ell\ge3).
$$
Proposition~\ref{pro:iskeleton-solutions} gives
$ y_1(u)=\frac{u(2-u)}{4}e^{-u}$ .
Then for $\ell=2$, we have
$ (e^u y_2(u))'=2e^u y_1(u)=\frac{u(2-u)}{2},$
and hence we get $y_2(u)=e^{-u}\left(\frac{u^2}{2!}-\frac{u^3}{3!}\right).$
Induction using $y_\ell'=y_{\ell-1}-y_\ell$ gives the solution
$ y_\ell(u)=e^{-u}\left(\frac{u^\ell}{\ell!}-\frac{u^{\ell+1}}{(\ell+1)!}\right)$.

$u=\sqrt2$ gives the component densities $\rho_\ell$'s. 
\end{proof}

The explicit formula has one further consequence that will be used repeatedly
when truncating component sizes in the terminal phase.

\begin{cor}
\label{cor:iskeleton-factorial-tail}
The limiting I-skeleton star distribution has a factorial tail.  More precisely,
for all $\ell\ge2$,
       $ 0\le \rho_\ell
        \le
        e^{-\sqrt2}\frac{(\sqrt2)^\ell}{\ell!}.$        
Consequently, for every fixed $q\ge0$, we have $\sum_{\ell\ge1}\ell^q\rho_\ell<\infty$. Then we can immediately get $\sum_{\ell\ge1}\rho_\ell\binom{\ell}{2}<\infty$ and $\sum_{\ell>L}\ell^q\rho_\ell\longrightarrow 0$ as $L \to \infty$
\end{cor}

\begin{proof}
The assertion follows immediately from the explicit formula in
Theorem~\ref{thm:intro-iskeleton}. For $\ell\ge2$,
       $ \rho_\ell
        =
        e^{-\sqrt2}
        \frac{(\sqrt2)^\ell}{\ell!}
        \left(1-\frac{\sqrt2}{\ell+1}\right).$
Since $1-\sqrt2/(\ell+1)\in[0,1]$ for $\ell\ge2$, we get
       $ 0\le \rho_\ell
        \le
        e^{-\sqrt2}\frac{(\sqrt2)^\ell}{\ell!}.$

The remaining claims follow from the convergence of the series.
\end{proof}

\section{From the I-skeleton process to the original process}
\label{sec:iskeleton-transfer}
In this section, we compare the original induced-$P_4$-free process $G$ with the auxiliary I-skeleton $\widehat G$. The I-skeleton process records the I--I and I--N edges that it adds, but it ignores edges between two non-isolated vertices. As isolatedness is evaluated in the current graph of each process, the same offered edge may have different types in $G_r$ and $\widehat G_r$. We contain the resulting discrepancies by marking every skeleton component affected by an N--N edge. The main result of this section is that the two processes agree exactly outside the marked region, while the marked vertices, marked pairs, and all fixed polynomial moments of marked original components are $o_{\mathbb P}(n/b_n^M)$ for every fixed $M$.

We first introduce the following deterministic parameters:
$$
        b_n=\left\lceil\frac{4\log n}{\log\log n}\right\rceil,
        \qquad
        A_n=(\log b_n)^2,
        \qquad
        r_n^{\max}=\lfloor A_n n/2\rfloor.
$$
Thus, $s_{r_n^{\max}}=2r_n^{\max}/n=A_n+O(n^{-1})$.

For $0\le r\le r_n^{\max}$, let $\widehat{\mathcal P}_r$ be the component
partition of $\widehat G_r$, including isolated vertices, and put
$
        \widehat R_r=\max_{C\in\widehat{\mathcal P}_r}|C|, \
        \widehat\Lambda_r=\frac{2}{n^2}\sum_{j<r}\widehat I_j.
$
Recall that $M_r=\binom{n}{2}-r$. For a graph \(G\), let
\(S(G):=\sum_{C\in\mathcal C(G)}|C|^2\)
denote its \emph{susceptibility}. We shall define the following stopping times:
$$
\sigma_\Lambda=\inf\{r:\widehat\Lambda_r>3\},\qquad
\sigma_{\mathrm{pair}}=\inf\{r:M_r<n^2/3\},
$$
$$
\sigma_R=\inf\{r:\widehat R_r>b_n\},\qquad
\sigma_S=\inf\{r:S(\widehat G_r)>50n\},
$$
and set
$$
        \sigma=\sigma_\Lambda\wedge\sigma_{\mathrm{pair}}
        \wedge\sigma_R\wedge\sigma_S .
$$

\subsection{Pre-stability of the auxiliary skeleton}

The argument is run until the slowly growing time
$r_n^{\max}=A_nn/2$.  We first strengthen the fixed-compact convergence from
Section~\ref{sec:iskeleton-process} to this interval.  The reason is that
$i(s)$ decays exponentially, so $A_n$ may grow slowly enough that all errors
remain smaller than any negative power of $b_n$.

\begin{lemma}
\label{lem:auxiliary-pre-stability}
Let $\mathbf x(s)=(i(s),c(s),y_1(s))$.  For every fixed $M>0$,
$$
        \sup_{0\le r\le r_n^{\max}}
        \left\|\widehat{\mathbf x}_r-\mathbf x(s_r)\right\|_\infty
        =o_{\mathbb P}(b_n^{-M}),
        \qquad
        \widehat I_{r_n^{\max}}=o_{\mathbb P}(n/b_n^M).
$$
Moreover, $\widehat\Lambda_{r_n^{\max}}\le3$ w.h.p.
\end{lemma}

\begin{proof}
Write
$
        \mathbf f(i,c,y)=
        \bigl(-i^2-i(c+y),\; \tfrac12 i^2,\; \tfrac12 i^2-2iy\bigr),
$
so that \(\mathbf x'(s)=\mathbf f(\mathbf x(s))\).  By
Proposition~\ref{pro:iskeleton-solutions},
$$
        i(u)=\frac{2-u^2}{2}e^{-u},\qquad \frac{du}{ds}=i(u),
        \qquad \sup_{s\to \infty}u(s)=\sqrt{2} .
$$
Since $i(u)=((\sqrt{2}+u)e^{-u}/2)(\sqrt{2}-u)$, \(i(u)\asymp \sqrt 2-u\), with $u^\prime=i$, there exist constants \(\alpha_1,\alpha_2>0\) such that
$i(s)\le \alpha_1e^{-\alpha_2s}$.
Hence, for every fixed \(M\),
$$
        b_n^M i(A_n)
        \le \alpha_1\exp\{M\log b_n-\alpha_2(\log b_n)^2\}\to0 .
$$
Also \(s_{r_n^{\max}}=A_n+O(n^{-1})\), and the same bound holds with
\(A_n\) replaced by \(s_{r_n^{\max}}\).

Now we compare the auxiliary process with the trajectory on the growing
interval \(0\le r\le r_n^{\max}\).  Let $K=[-1/10,2]^3$
and let \(\theta\) be the first time when
\(\widehat{\mathbf x}_r=(\widehat i_r,\widehat c_r,\widehat y_{1,r})\) leaves
\(K\).  The deterministic trajectory stays a fixed positive distance from the boundary of $K$ for all \(0\le s\le A_n\).  Moreover, \(\mathbf f\) is Lipschitz
on \(K\).

For \(r\le r_n^{\max}\), we have
$$
        M_r=\frac{n^2}{2}\bigl(1+O(A_n/n+n^{-1})\bigr),
        \qquad
        \frac1{M_r}=\frac2{n^2}\bigl(1+O(A_n/n+n^{-1})\bigr).
$$
Thus the counting estimates from
Proposition~\ref{pro:closed-iskeleton-evolution}, applied directly on this
slowly growing interval, give uniformly for \(r\le r_n^{\max}\wedge\theta\),
$$
        \mathbb E\!\left[
        \widehat{\mathbf x}_{r+1}-\widehat{\mathbf x}_{r}
        \mid \mathcal F_r
        \right]
        =
        \frac{2}{n}\mathbf f(\widehat{\mathbf x}_r)
        +
        O\!\left(\frac{A_n/n+n^{-1}}{n}\right).
$$
Consequently, for the stopped process,
$$
        \widehat{\mathbf x}_{r\wedge\theta}
        =
        \widehat{\mathbf x}_{0}
        +\sum_{j<r\wedge\theta}\frac{2}{n}
             \mathbf f(\widehat{\mathbf x}_j)
        +\mathbf M_r
        +\mathbf E_r ,
$$
where \((\mathbf M_r)\) is a vector martingale and
$\sup_{r\le r_n^{\max}}\|\mathbf E_r\|
        =
        O\!\left({A_n^2}/n\right)$.
Each martingale increment is \(O(n^{-1})\), and the total predictable quadratic
variation of each coordinate up to \(r_n^{\max}\) is \(O(A_n/n)\).  By
Lemma~\ref{lem:freedman} and a union bound over the three coordinates,
$$
        \sup_{r\le r_n^{\max}}\|\mathbf M_r\|_\infty
        =
        O\!\left(\sqrt{\frac{A_n\log n}{n}}\right).
$$

On the other hand, the deterministic solution satisfies the following expansion
$$
        \mathbf x(s_{j+1})-\mathbf x(s_j)
        =
        \frac{2}{n}\mathbf f(\mathbf x(s_j))
        +O(n^{-2})
$$
uniformly for \(0\le j<r_n^{\max}\).  Since \(\mathbf f\) is Lipschitz on \(K\),
the previous two displays and the discrete Gronwall inequality imply that
$$
        \Delta_n:=
        \sup_{0\le r\le r_n^{\max}}
        \bigl\|
        \widehat{\mathbf x}_{r\wedge\theta}-\mathbf x(s_r)
        \bigr\|_\infty
        \le
        e^{CA_n}
        \left[
        O\!\left(\sqrt{\frac{A_n\log n}{n}}\right)
        +
        O\!\left(\frac{A_n^2}{n}\right)
        \right].
$$
Since \(A_n=(\log b_n)^2=O((\log\log n)^2)\), the right-hand side is
\(o(b_n^{-M})\) for every fixed \(M\).  In particular
\(\Delta_n=o_{\mathbb P}(1)\).  Because the deterministic trajectory has a
fixed positive distance inside \(K\), it follows that
$
        \Pr(\theta\le r_n^{\max})\to0 
$.
Hence the same comparison estimate holds without stopping:
$$
        \sup_{0\le r\le r_n^{\max}}
        \left\|\widehat{\mathbf x}_r-\mathbf x(s_r)\right\|_\infty
        =
        o_{\mathbb P}(b_n^{-M}).
$$

Applying this to the first coordinate and using
\(i(s_{r_n^{\max}})=o(b_n^{-M})\), we obtain
$$
        \widehat I_{r_n^{\max}}
        =
        n\widehat i_{r_n^{\max}}
        =
        o_{\mathbb P}(n/b_n^M).
$$

It remains to control \(\widehat\Lambda\). By definition,
the comparison estimate and the Riemann-sum error give that
$$
        \left|
        \widehat\Lambda_{r_n^{\max}}
        -
        \int_0^{s_{r_n^{\max}}}i(s)\,ds
        \right|
        \le
        A_n\Delta_n+O(A_n/n)
        =
        o_{\mathbb P}(1).
$$
But
$$
        \int_0^{s_{r_n^{\max}}}i(s)\,ds
        =
        u(s_{r_n^{\max}})
        \le \sqrt2<3 .
$$
Therefore \(\widehat\Lambda_{r_n^{\max}}\le3\) w.h.p.
\end{proof}

The preceding lemma establishes uniform control of the first-order state variables over the pre-stability interval. To carry out the component-marking argument, we additionally require uniform bounds on the sizes of the skeleton components, together with a corresponding susceptibility bound, throughout the same interval.

\begin{lemma}
\label{lem:auxiliary-component-bounds}
For the auxiliary skeleton process,
$$
        \Pr\!\left(\sup_{0\le r\le r_n^{\max}}\widehat R_r>b_n\right)=o(1),
        \qquad
        \Pr\!\left(\sup_{0\le r\le r_n^{\max}}
        S(\widehat G_r)>50n\right)=o(1).
$$
Moreover, for every fixed $q\ge1$ and every $\eta>0$,
$$
\lim_{L_0\to\infty}\limsup_{n\to\infty}
\Pr\!\left(
\frac1n\sum_{\substack{C\in\widehat{\mathcal P}_{r_n^{\max}}\\ |C|>L_0}}
|C|^q>\eta
\right)=0.
$$
Consequently $\Pr(\sigma\le r_n^{\max})=o(1)$.
\end{lemma}
\begin{proof}
Set $\tau_0=r_n^{\max}\wedge\sigma_\Lambda\wedge\sigma_{\mathrm{pair}}$.
We first prove the component-size estimates for the process stopped at
\(\tau_0\).  Consider a component at the moment it is created by an accepted I--I
edge. At every later time before \(\tau_0\), it has at most two
universal vertices to which an isolated vertex can connect.  Since
\(M_r\ge n^2/3\) for \(r<\tau_0\), the conditional probability that the next edge attaches to this particular descendant is at most
${2\widehat I_r}/{M_r}\le{6\widehat I_r}/{n^2}$.
Consequently, the total conditional intensity of all later attachments to this
component before \(\tau_0\) is at most
$$
        \sum_{r<\tau_0}\frac{6\widehat I_r}{n^2}
        \le
        3\widehat\Lambda_{\tau_0}
        \le 9+\frac{6}{n}\leq 10 .
$$

We use the standard exponential domination for adapted Bernoulli trials in its
stopped form.  If \(B_C\) is the number of later attachments to any fixed
stopping-time-born component $C$ before \(\tau_0\), then, conditionally on its
birth history,
$$
        \mathbb E\!\left[e^{\theta B_C}\mid \text{birth history}\right]
        \le
        \exp\{10(e^\theta-1)\},
        \qquad \theta>0 .
$$
Consequently, \(B_C\) satisfies the Poisson-type Chernoff bound with mean $10$. In particular, for every \(\beta\ge1\),
$$
        \Pr(B_C\ge \beta\mid \text{birth history})
        \le
        \inf_{\theta>0}
        \exp\{-\theta \beta+10(e^\theta-1)\}.
$$
Taking \(\beta=b_n-2\) and $\theta=\log(({b_n-2})/{10})$ for all large \(n\), gives
$$
        \Pr(B_C\ge b_n-2)
        \le
        \exp\{-(1+o(1))b_n\log b_n\}
        =
        n^{-4+o(1)} .
$$
There are at most \(n/2\) seed components, so a union bound gives
$\Pr\!\left(\sup_{r\le \tau_0}\widehat R_r>b_n\right)=o(1)$.

The same domination gives the polynomial tail.  Since each non-trivial skeleton
component has size \(2+B_C\), and the number of non-trivial components is at most \(n/2\),
for every fixed \(q\ge1\),
$$
 \frac1n
 \mathbb E\sum_{\substack{C\in\widehat{\mathcal P}_{\tau_0}\\ |C|>L_0}}
        |C|^q
 \le
        \frac12\,
        \mathbb E\!\left[(2+B_C)^q\mathbf 1_{\{2+B_C>L_0\}}\right] 
 \le
        \gamma_q\sum_{\beta>L_0/2} \beta^{q-1}
        \Pr(B_C\ge \beta),
$$
where the last bound is the usual tail-sum estimate.  The right-hand side tends
to \(0\) as \(L_0\to\infty\), uniformly in \(n\), because the displayed
Poisson-type tail is exponential.  Markov's inequality then gives
$$
\lim_{L_0\to\infty}\limsup_{n\to\infty}
\Pr\!\left(
\frac1n\sum_{\substack{C\in\widehat{\mathcal P}_{\tau_0}\\ |C|>L_0}}
|C|^q>\eta
\right)=0 .
$$
Lemma~\ref{lem:auxiliary-pre-stability} gives
\(\Pr(\tau_0<r_n^{\max})=o(1)\), and hence the maximum-size and polynomial-tail
claims hold at \(r_n^{\max}\) without stopping.

It remains to prove the susceptibility estimate.  Stop now at $
        T=r_n^{\max}\wedge\sigma_\Lambda
        \wedge\sigma_{\mathrm{pair}}\wedge\sigma_R
$.  For \(r<T\), an accepted I--I edge
changes \(S(\widehat G_r)\) by $2$,
an accepted I--N attachment to a component \(C\) changes \(S(\widehat G_r)\) by $2|C|$.
The conditional probability of an I--I edge is at most
$
        \binom{\widehat I_r}{2}/{M_r}\le{3\widehat I_r^2}/{2n^2},
$
and the conditional probability of an I--N attachment to a fixed non-trivial
component \(C\) is at most
${2\widehat I_r}/{M_r}\le{6\widehat I_r}/{n^2}$
because a star has at most two universal vertices.  Therefore
$$
 \mathbb E[\Delta S(\widehat G_r)\mid\mathcal F_r]
 \le
        2\cdot\frac{3\widehat I_r^2}{2n^2}
        +
        \sum_{C\in\widehat{\mathcal P}_r:\, |C|>1}
        2|C|\cdot\frac{6\widehat I_r}{n^2} 
 \le
        3\frac{\widehat I_r}{n}
        +
        12\frac{\widehat I_r}{n}
        \le
        18\frac{\widehat I_r}{n}.
$$
Here we used \(\widehat I_r\le n\) and
\(\sum_{C:|C|>1}|C|\le n\).  Hence the total predictable drift before \(T\) is
at most
$$
        \sum_{r<T}18\frac{\widehat I_r}{n}
        \le
        18\frac1n\cdot (\frac{3n^2}{2}+n)
        =27n+18 .
$$
We shall use the  bound \(36n\).

Let
$$
        M_r^S = S(\widehat G_{r\wedge T})-S(\widehat G_0)
        - \sum_{j<r\wedge T} \mathbb E[\Delta S(\widehat G_j)\mid\mathcal F_j]
$$
be the martingale part.  Since \(r<T\) implies \(\widehat R_r\le b_n\), every
increment of \(S(\widehat G_r)\) before the stop is at most \(2b_n\); we use the bound
\(4b_n\).  Moreover, because the increments are non-negative,
$$
        \mathbb E[(\Delta S(\widehat G_r))^2\mid\mathcal F_r]
        \le
        4b_n\,\mathbb E[\Delta S(\widehat G_r)\mid\mathcal F_r],
$$
and so the predictable quadratic variation of \(M^S\) up to \(T\) is at most
$4b_n\cdot 36n=144nb_n$.
Freedman's inequality therefore gives
$$
        \Pr\!\left(M_T^S>12n\right)
        \le
        \exp\!\left\{
        -\frac{(12n)^2}{2(144nb_n+4b_n\cdot12n/3)}
        \right\}
        =o(1).
$$
Since \(S_0=n\), we get
$
        S_T
        \le
        n+36n+12n
        \le
        50n
$
w.h.p.  Together with the maximum-size estimate already proved,
this implies
$$
        \Pr\!\left(
        \sup_{0\le r\le r_n^{\max}}S(\widehat G_r)>50n
        \right)=o(1).
$$
Combining this with
\(\Pr(\sigma_\Lambda\le r_n^{\max})=o(1)\), the deterministic fact
\(\sigma_{\mathrm{pair}}>r_n^{\max}\) for all large \(n\), and the
maximum-size estimate, we obtain
$\Pr(\sigma\le r_n^{\max})=o(1)$.
\end{proof}

\subsection{The marked region and the coupling invariant}
With the auxiliary skeleton process determined, we can now compare it with the original process. The possible discrepancies are exact inside a marked region, and the
probability estimates show that this marked region remains tiny.

Denote all marked components in $\widehat{G}_r$ by $\widehat B_r$. Initialize $\widehat B_0=\varnothing$. Let $e_{r+1}=xy$. We define the evolution of the marked region by the following rules.

\begin{enumerate}[label=(B\arabic*)]
\item If $x$ and $y$ lie in the same unmarked non-trivial I-skeleton process
component $C$, mark every vertex of $C$.
\item If one endpoint lies in the marked region and the other lies in a
unmarked I-skeleton component $C$, mark every vertex
of $C$.
\end{enumerate}

Rule (B1) marks a good non-trivial skeleton process
component when an internal N-N edge may create a discrepancy between $G$ and
$\widehat G$.  Rule (B2) absorbs any good component touched by an already marked
component, ensuring that discrepancies cannot leak back into the good region.

When \(e_r\) is offered to the process, we first apply the marking rules (B1) and (B2). Let \(\widehat B_r^-\) denote the resulting marked set, immediately before the graph updates. After the corresponding updates of the original process and the I-skeleton process have been performed, we denote the marked set by \(\widehat B_r\). The graph updates themselves do not introduce any additional marked vertices, so \(\widehat B_r=\widehat B_r^-\) as vertex sets; the distinction in notation merely indicates whether the graph update at step \(r\) has already been performed.

The following deterministic lemma shows that the marked region contains all discrepancies between the original process and the auxiliary skeleton process.

\begin{lemma}
\label{lem:coupling-invariant}
For every $r\ge0$, the following assertions hold.
\begin{enumerate}[label=(\roman*)]
\item $\widehat B_r$ is a union of connected components of both $\widehat G_r$
and $G_r$.
\item Every unmarked component $C$ of $\widehat G_r$ is the vertex set of a
connected component of $G_r$.
\item For every unmarked component $C$, $G_r[V(C)]=\widehat G_r[V(C)]$.
\item Conversely, every connected component of $G_r$ disjoint from
$\widehat B_r$ corresponds to a unique unmarked component of
$\widehat G_r$.
\end{enumerate}
\end{lemma}

\begin{proof}
We prove the assertions by induction on \(r\).  They are immediate at \(r=0\).
Assume that the assertions hold at time \(r-1\), and let
\(e_r=xy\) be the next offered edge.  By the induction hypothesis and
the whole-component marking rules, \(\widehat B_r^-\) is a union of
components of both \(G_{r-1}\) and \(\widehat G_{r-1}\).

First, we offer the edge $xy$ and apply the marking rules. If there is at least one endpoint in \(\widehat B^-_r\), say $x$, then the component containing $y$ has already been marked, or has just been marked at this step by (B2). Thus, the update cannot create a discrepancy on any unmarked component, and every possible discrepancy remains inside \(\widehat B^-_r\).

Then we update the graph. It remains only to check the case in which both endpoints lie outside \(\widehat B^-_r\). 
By (B1), they do not lie in the same non-trivial good component. Hence, the offer is one of
the three types: I--I, I--N with a singleton and a non-trivial good component, or N--N with two distinct non-trivial unmarked components. In all three cases the original update agrees with the auxiliary
update on \(V\setminus \widehat B^-_r\), by the defining rules of the auxiliary skeleton
and Proposition~\ref{pro:structural-invariants}(i),(iii).

As the update does not change anything in the unmarked components, the same component closure holds in the original graph. Therefore, the induced graphs and the component correspondence remain identical outside \(\widehat B_{r+1}\), completing the induction.
\end{proof}

 If a component is unmarked, we call it a good component at the current stage. It remains to bound the size of the marked component generated by the marking rules. We set
$$
        Z_r=|\widehat B_r|,
        \qquad
        Q_r=\sum_{\substack{C\in\widehat{\mathcal P}_r:\,C\subseteq\widehat B_r}}
        \binom{|C|}{2}.
$$
For fixed $q\ge1$, also put
$$
        R_{q,r}=\sum_{\substack{D\in\mathcal C(G_r):\,V(D)\subseteq\widehat B_r}}
        |D|^q.
$$

\begin{lemma}
\label{lem:marked-region-control}
For every fixed $M>0$,
$$
        \sup_{0\le r\le r_n^{\max}}Z_r=o_{\mathbb P}(n/b_n^M),
        \qquad
        \sup_{0\le r\le r_n^{\max}}Q_r=o_{\mathbb P}(n/b_n^M).
$$
Moreover, for every fixed $q\ge1$ and $M>0$,
$$
        \sup_{0\le r\le r_n^{\max}}R_{q,r}=o_{\mathbb P}(n/b_n^M).
$$
\end{lemma}

\begin{proof}
We work first up to $r_n^{\max}\wedge\sigma$, where
$\widehat R_r\le b_n$, $M_r\ge n^2/3$, and
$S(\widehat G_r)\le50n$.  The marked region is a union of current
I-skeleton components, so every skeleton component is either contained in the
marked region or wholly unmarked. Conditional on \(\mathcal F_r\), the next offered pair is
uniform over the exact \(M_r\) unrevealed pairs.

For rule (B1), the next edge must have both endpoints in the same good
non-trivial skeleton component. Since \(M_r\ge n^2/3\), the conditional expected number of vertices added is at most
\[
        \sum_{C\text{ good},\ |C|\ge2}|C|\frac{\binom{|C|}{2}}{M_r}
        \le \frac12\sum_{C\text{ good}} |C|\frac{|C|^2}{M_r}
        \le \frac{b_n}{2M_r} \sum_{C\text{ good}} |C|^2 \le \frac {50nb_n}{n^2/3}
        \le \gamma_1\frac{b_n}{n},
\]
where $\gamma_1$ is a constant.

For rule (B2), the conditional expected number of vertices added is at most
\[
        \sum_{C\text{ good}}|C|\frac{2Z_r|C|}{M_r} =\frac{2Z_r}{M_r}\sum_{C\text{ good}}|C|^2
        \le \gamma_2\frac{Z_r}{n},
\]
where $\gamma_2$ is a constant.

Adding the two possible marking contributions, we obtain,
\[
        \mathbb E[Z_{(r+1)\wedge\sigma}\mid\mathcal F_r]
        \le\left(1+\frac{\gamma_2}{n}\right)Z_{r\wedge\sigma}
        +\gamma_1\frac{b_n}{n}.
\]
Iterating over $O(A_nn)$ steps with Lemma~\ref{lem:discrete-gronwall} gives
\[
        \mathbb E Z_{r_n^{\max}\wedge\sigma}
        \le \gamma_3 e^{\gamma_3 A_n}A_nb_n,
\]
where $\gamma_3$ is a constant.

Since $A_n=(\log b_n)^2$ and $b_n\asymp \log n/\log\log n$, we have
$e^{\gamma_3A_n}A_nb_n=n^{o(1)}=o(n/b_n^M)$ for every fixed $M$.  As
$Z_{r\wedge\sigma}$ is non-decreasing,
\[
\Pr\!\left(\sup_{r\le r_n^{\max}}Z_r>\frac{n}{b_n^M}\right)
\le
\Pr(\sigma\le r_n^{\max})
+
\Pr\!\left(Z_{r_n^{\max}\wedge\sigma}>\frac{n}{b_n^M}\right)
=o(1),
\]
by Markov's inequality and Lemma~\ref{lem:auxiliary-component-bounds}.  This
proves the first assertion.

On $\{\sigma>r_n^{\max}\}$, every marked component has size at most
$b_n$, and hence $Q_r\le b_nZ_r/2$ for all $r\le r_n^{\max}$.  Therefore
\[
\Pr\!\left(\sup_{r\le r_n^{\max}}Q_r>\frac{n}{b_n^M}\right)
\le
\Pr(\sigma\le r_n^{\max})
+
\Pr\!\left(\sup_{r\le r_n^{\max}}Z_r>\frac{2n}{b_n^{M+1}}\right)
=o(1),
\]
which proves the second assertion.

Fix $q\ge1$. For a graph $H$ and a set $B$ which is a union of
components of $H$, write
\[
        R_q(H,B):=\sum_{\substack{D\in\mathcal C(H):\,V(D)\subseteq B}} |D|^q .
\]
Thus $R_{q,r}=R_q(G_r,\widehat B_r)$.  We again estimate one step on the event
$r<\sigma$.

Let $xy$ be the next offered edge. By Lemma~\ref{lem:coupling-invariant}, \(\widehat B^-_r\) is a union of current original components, and after the update it remains a union of original components. Moreover, all possible discrepancies created at this step are confined to \(\widehat B^-_r\). Hence
\[
        R_{q,r}-R_{q,r-1}
        =
        \bigl(R_q(G_{r-1},\widehat B^-_r)-R_q(G_{r-1},\widehat B_{r-1})\bigr)
        +
        \bigl(R_q(G_{r},\widehat B_r)-R_q(G_{r-1},\widehat B^-_r)\bigr).
\]
The first term records the contribution of old original components that become marked before the update, while the second records the contribution of the graph update inside the already enlarged marked set.

We first bound the marking term. The term \(R_q(G_{r-1},\widehat B^-_r)-R_q(G_{r-1},\widehat B_{r-1})\) is controlled by the same two marking mechanisms as \(Z_r\). A (B1) marking can occur only when the offered pair lies inside an unmarked non-trivial component $C$, whose increase is $|C|^q$. Thus, the conditional expected contribution of (B1) is at most
\[
        \frac1{M_{r-1}}\sum_{C\text{ good},\,|C|\ge2}|C|^q\binom{|C|}{2}
        \le \gamma_q\frac{b_n^q}{n},
\]
where $\gamma_q$ is a constant and the estimate uses $|C|\le b_n$, $M_{r-1}\ge n^2/3$, and $\sum_C |C|^2\le50n$.

For (B2), suppose that an unmarked component \(C\) is absorbed. The insertion contributes \(|C|^q\), and the conditional probability of this fixed \(C\) is at most \(2Z_{r-1}|C|/M_{r-1}\).  Hence, the total conditional contribution of (B2) is at most
\[
        \frac{2Z_{r-1}}{M_{r-1}}\sum_{C\text{ good}} |C|^{q+1}
        \le
        \frac{2Z_{r-1}}{M_{r-1}} b_n^{q-1}\sum_{C\in\widehat{\mathcal P}_{r-1}}|C|^2
        \le
        \gamma_q\frac{b_n^{q-1}Z_{r-1}}{n}.
\] 
Therefore
\[
        \mathbb E\!\left[
        R_q(G_{r-1},\widehat B^-_r)-R_q(G_{r-1},\widehat B_{r-1})
        \,\middle|\,\mathcal F_{r-1}
        \right]
        \le
        \gamma_q\frac{b_n^q}{n}
        +\gamma_q\frac{b_n^{q-1}Z_{r-1}}{n}.
\]

It remains to bound the update term $R_q(G_{r},\widehat B_r)-R_q(G_{r-1},\widehat B^-_r)$. After (B1)--(B2) have been applied, the offered pair has either both endpoints in \(\widehat B^-_r\) or both endpoints in \(V\setminus \widehat B^-_r\).  In the latter case the value of \(R_q(G_r,\widehat B^-_r)\) is unchanged.  If (B1) is responsible for the enlargement, then the offered pair lies inside a single good component which has just been inserted into \(\widehat B^-_r\), and again the component partition inside \(\widehat B^-_r\) does not change.

Thus, a non-zero update contribution can only come from an accepted component merger inside \(\widehat B^-_r\) involving the old marked region. By Proposition~\ref{pro:structural-invariants}(i)(iii), every such accepted merger has at least one isolated endpoint.  Hence, it is either an I--I merger or an I--N attachment.

The I--I contribution is at most
\[
        \gamma_q\frac{Z_{r-1} n}{M_{r-1}}
        \le  \gamma_q\frac{Z_{r-1}}{n}
        \le \gamma_q\frac{R_{q,r}}{n}.
\]

For an I--N attachment, let \(D\) be the non-trivial component involved. If \(|D|=x\), the increase in \(R_q\) is at most $(x+1)^q-x^q-1\le \gamma_q(x^{q-1}+1)$. If $D$ is marked before $xy$ is offered, then summing over the isolated endpoint gives at most
\[
\begin{aligned}
        \frac{\gamma_q n}{M_{r-1}}
        \sum_{\substack{D\in\mathcal C(G_r):\,
        V(D)\subseteq\widehat B_r}}
        |D|\bigl(|D|^{q-1}+1\bigr)
        &\le
        \gamma_q \frac{R_{q,r}}{n}.
\end{aligned}
\]
If \(D\) was previously unmarked, then the isolated endpoint must already lie in
\(\widehat B_r\), and \(D\) is absorbed by (B2) before the update.  Hence
\[
\begin{aligned}
        \frac{\gamma_qZ_r}{M_r}
        \sum_{\substack{D\text{ good}\\ |D|\ge2}}
        |D|\bigl(|D|^{q-1}+1\bigr)
        &\le
        \gamma_q\frac{b_n^{q-1}Z_r}{n}.
\end{aligned}
\]
Therefore
\[
        \mathbb E\!\left[
        R_q(G_{r},\widehat B^-_r)-R_q(G_{r-1},\widehat B^-_r)
        \,\middle|\,\mathcal F_r
        \right]
        \le
        \gamma_q\frac{R_{q,r}}{n}
        +
        \gamma_q\frac{b_n^{q-1}Z_r}{n}.
\]

Combining the marking and update bounds yields
\[
\mathbb E[R_{q,(r+1)\wedge\sigma}\mid\mathcal F_r]
\le
\left(1+\frac{\gamma_q}{n}\right)R_{q,r\wedge\sigma}
+\gamma_q\frac{b_n^q}{n}
+\gamma_q\frac{b_n^{q-1}Z_{r\wedge\sigma}}{n}.
\]
Using the already proved bound
$\sup_{r\le r_n^{\max}}\mathbb E Z_{r\wedge\sigma}\le Ce^{CA_n}A_nb_n$ and
iterating this recursion over $O(A_nn)$ steps with Lemma~\ref{lem:discrete-gronwall} gives
\[
        \mathbb E R_{q,r_n^{\max}\wedge\sigma}
        \le \gamma_q e^{\gamma_q A_n}
        \left(A_nb_n^q+e^{CA_n}A_n^2b_n^q\right)
        =n^{o(1)}.
\]
In particular this is $o(n/b_n^M)$ for every fixed $M$.  Finally, since
$R_{q,r\wedge\sigma}$ is non-decreasing,
\[
\Pr\!\left(
        \sup_{r\le r_n^{\max}}R_{q,r}>\frac{n}{b_n^M}
\right)
\le
\Pr(\sigma\le r_n^{\max})
+
\Pr\!\left(
        R_{q,r_n^{\max}\wedge\sigma}>\frac{n}{b_n^M}
\right)
=o(1),
\]
again by Lemma~\ref{lem:auxiliary-component-bounds} and Markov's inequality, which completes the proof.
\end{proof}

\subsection{Skeleton-to-process transfer}

We now combine the deterministic coupling invariant with the marked-region estimates to obtain a transfer principle from the auxiliary skeleton process to the original process. This result will be used in the subsequent analysis of the terminal graph.

For \(ab\in\{\mathrm{II},\mathrm{IN}\}\), let \(X_{ab}^{G}(r)\)
denote the number of edges among \(e_1,\ldots,e_r\) that are accepted in
the original process and have type \(ab\) immediately before they are
offered. 

\begin{thm}
\label{thm:skeleton-comparison}
Couple $G$ and $\widehat G$ by the same uniformly random edge ordering.  For
every $r\le r_n^{\max}$, the two processes agree componentwise and edgewise
outside $\widehat B_r$. Then, for every fixed connected $F$ and $M>0$,
$$
        \sup_{0\le r\le r_n^{\max}}
        |C_F(G_r)-C_F(\widehat G_r)|
        =o_{\mathbb P}(n/b_n^M).
$$
Moreover,
\[
\sup_{0\le r\le r_n^{\max}}
\left(
|X_{\mathrm{II}}^{G}(r)-\widehat X_{\mathrm{II}}(r)|
+
|X_{\mathrm{IN}}^{G}(r)-\widehat X_{\mathrm{IN}}(r)|
\right)
=o_{\mathbb P}(n/b_n^M).
\]
Finally, for every fixed \(q\ge1\),
\[
\sup_{0\le r\le r_n^{\max}}
\sum_{\substack{D\in\mathcal C(G_r):\,
V(D)\subseteq\widehat B_r}}
|D|^q
=o_{\mathbb P}(n/b_n^M).
\]
\end{thm}

\begin{proof}
Lemma~\ref{lem:coupling-invariant} localizes all discrepancies between the
true process and the skeleton process to the marked region.  Consequently, the error in every fixed component count is \(O(Z_r)\),
the error in the accepted I--I and I--N counts is \(O(Z_r+Q_r)\), and
the final assertion is exactly the \(R_{q,r}\)-bound in
Lemma~\ref{lem:marked-region-control}.  Applying that lemma proves all
three assertions.
\end{proof}

As a first result, the I--I and I--N edge counts in the skeleton process have the same first-order limits in the original process. For $ab\in\{\mathrm{II},\mathrm{IN}\}$, write $X_{ab}:=X_{ab}^{G}(N)$ for the corresponding terminal counts.

\begin{pro}\label{pro:iskeleton-edge-counts}
For the original induced-$P_4$-free process,
$$
        \frac{X_{\mathrm{II}}}{n}\xrightarrow{\mathbb P}\beta,
        \qquad
        \frac{X_{\mathrm{IN}}}{n}\xrightarrow{\mathbb P}1-2\beta,
$$
Consequently,
$$
        \frac{X_{\mathrm{II}}+X_{\mathrm{IN}}}{n}\xrightarrow{\mathbb P}1-\beta.
$$
\end{pro}

\begin{proof}
The limits in the auxiliary process are given in Section~\ref{sec:iskeleton-process}. By Lemma~\ref{lem:auxiliary-pre-stability}, $\widehat I_{r_n^{\max}}=o_{\mathbb P}(n/b_n^2)$, and Theorem~\ref{thm:skeleton-comparison} gives $I_{r_n^{\max}}(G)=o_{\mathbb P}(n/b_n^2)$. Up to time $r_n^{\max}$, accepted I-I and I-N counts in the two processes differ by $o_{\mathbb P}(n)$. After $r_n^{\max}$, every further accepted I-I or I-N edge consumes a vertex that is isolated immediately before that step, and accepted edges never create isolated vertices. After \(r_n^{\max}\), the total further contribution is \(o_{\mathbb P}(n)\), so the same limits hold for the original process.
\end{proof}

\section{From skeleton stars to terminal components}\label{sec:terminal-phase}

We now transfer the star decomposition of the auxiliary process to the terminal component
structure of the original process. The argument proceeds in two steps. First,
we stop the process once the number of isolated vertices becomes negligible
and discard a collection of exceptional components whose contribution to
every fixed polynomial moment is negligible. Second, conditional on the
history up to the stopping time, the unexposed leaf--leaf pairs within each
remaining skeleton star retain a uniformly random relative order. Moreover,
these relative orders are independent between distinct stars. Consequently,
the subsequent evolution on each leaf set is distributed as an independent
finite induced-\(P_4\)-free process.

\subsection{Stopping time and clean stars}

By Lemma~\ref{lem:auxiliary-pre-stability} and
Theorem~\ref{thm:skeleton-comparison}, for every fixed integer $m\ge1$,
$b_n^m{I_{r_n^{\max}}(G)}=o_{\mathbb P}(n)$.    
Applying Lemma~\ref{lem:deterministic-envelope} with
$W_n=I_{r_n^{\max}}(G)/n$, we obtain a deterministic sequence
$\varepsilon_n$  such that
\begin{equation}\label{eq:epsilon-envelope}
        \varepsilon_n\longrightarrow0,\qquad
        \frac{I_{r_n^{\max}}(G)}{\varepsilon_n n}
        \xrightarrow{\mathbb P}0,
        \quad
        \varepsilon_n b_n^m\longrightarrow0
        \quad\text{for every fixed }m\ge1.
\end{equation}
We introduce the stopping time $\tau_1=\inf\{r:I(G_r)\le\varepsilon_n n\}$. Then $\tau_1\le r_n^{\max}$ w.h.p.  Since accepted edges do not create isolated vertices, every accepted I--I or I--N edge after $\tau_1$
consumes at least one of the vertices isolated at time $\tau_1$. Consequently, there are at most $I_{\tau_1}(G)\le\varepsilon_n n$ accepted I--I or I--N edges after $\tau_1$.

By Theorem~\ref{thm:skeleton-comparison}, each unmarked star \(S\) in
\(\widehat G_{\tau_1}\) is the same labelled star in \(G_{\tau_1}\),
with the same vertex set, center, and leaf set.

The stars that are easy to analyze must satisfy the following three conditions: they lie outside the marked component, obey the component-size localization, and have no leaf--leaf pair exposed before $\tau_1$.

\begin{definition}\label{def:tau1-clean}
A non-trivial I-skeleton component
$S$ is called \emph{$\tau_1$-clean} if
\begin{enumerate}[label=(P\arabic*)]
\item $S\cap\widehat B_{\tau_1}=\varnothing$;
\item $S$ is a star with at most $b_n-1$ leaves;
\item no pair of leaves of $S$ has been offered by time $\tau_1$.
\end{enumerate}
\end{definition}

A star present at $\tau_1$ that fails one of (P1)--(P3) will be called, respectively, marked, size-exceptional, or pre-exposed. A clean star is called \emph{post-$\tau_1$-touched} if its represented component later receives an accepted I-N attachment, and a component is called \emph{post-$\tau_1$-born} if it is created by an I-I edge after $\tau_1$.

\subsection{Exceptional component and post-stability evolution}

We first record the persistence property that for an isolated vertex between successive I--N attachments.

\begin{lemma}\label{lem:persistent-centres}
Let \(C\) be a non-trivial component at time \(s\), and fix \(u\in U(C)\). Until the component containing \(C\) next receives an accepted I--N edge, \(u\) remains universal. Consequently, if \(v\) is isolated and $uv$ is unoffered at time $s$, then \(vu\) remains acceptable until the earlier of that I--N attachment and the time at which $v$ ceases to be isolated.
\end{lemma}

\begin{proof}
Before the next accepted I--N attachment, the vertex set of the component containing \(C\) does not change, and its only accepted updates add internal edges. Thus \(u\) remains universal. While $v$ remains isolated,
Proposition~\ref{pro:structural-invariants}(iii) shows that $uv$ is acceptable.
\end{proof}

First, we estimate stars whose future leaf set has already been partially exposed.

\begin{lemma}
\label{lem:preexposed-stars}
Let $\mathcal P_n^{\mathrm{pre}}$ be the family of unmarked I-skeleton
components present at $\tau_1$ for which at least one pair of leaves has been
offered before $\tau_1$.  For every fixed $q\ge1$ and $M>0$,
\[
        \sum_{S\in\mathcal P_n^{\mathrm{pre}}}|S|^q
        =o_{\mathbb P}(n/b_n^M).
\]
\end{lemma}

\begin{proof}
We work on the event \(\{\tau_1\le r_n^{\max},\,\sigma>r_n^{\max}\}\), whose probability tends to \(1\).

Fix \(S_{\tau_1}\in\mathcal P_n^{\mathrm{pre}}\), and consider the first
edge offered before \(\tau_1\) whose endpoints later become leaves of
\(S_{\tau_1}\). While the relevant component remains unmarked,
Lemma~\ref{lem:coupling-invariant},
Proposition~\ref{pro:structural-invariants}, and rule (B1) leave only
the following possibility: at some step \(r+1\), an edge \(xy\) is
offered from an isolated vertex \(x\) to a proper leaf \(y\) of an
unmarked star \(S\), the edge \(xy\) is rejected, and \(x\) later
attaches to the center of the component containing \(S\). The \(K_2\)
case causes no difficulty, since a \(K_2\) has no proper leaf before a
later attachment determines a unique center.

Fix such an offer \(xy\) at step \(r+1\). Let \(\mathcal A_r\) be the
family of current unmarked star components with at least two leaves.
For each \(S'\in\mathcal A_r\), let \(u(S')\) denote its unique center.
Since \(x\) is isolated at time \(r\), the pair \(xu(S')\) is still
unoffered. Moreover, while \(S'\) remains unmarked,
Lemma~\ref{lem:persistent-centres} ensures that \(u(S')\) remains
universal, so \(xu(S')\) is an available I--N attachment whenever
\(x\) is still isolated.

Let \(P_{xS}\) be the event that \(x\) later attaches to the component
containing \(S\) and that the resulting component remains unmarked.
On \(P_{xS}\), the pair \(xu(S)\) must be the first pair in \(\{xu(S'):S'\in\mathcal A_r\}\)
to appear in the remaining edge order. Indeed, if \(xu(S')\) with
\(S'\ne S\) appears first while \(S'\) is unmarked, then \(x\) attaches
to \(S'\). If \(S'\) has already become marked, rule (B2) marks the
component containing \(x\), which again prevents the later component
containing \(S\) from remaining unmarked. Hence, by
Lemma~\ref{lem:suffix-exchangeability},
\[
        \mathbb P(P_{xS}\mid\mathcal F_r)
        \le \frac{1}{|\mathcal A_r|}.
\]

We now count such successful side offers. Conditional on
\(\mathcal F_r\), the next offered pair is uniform among the \(M_r\)
unrevealed pairs. For each \(S\in\mathcal A_r\), there are at most
\(\widehat I_r(|S|-1)\) choices of a pair \(xy\), where \(x\) is
isolated and \(y\) is a proper leaf of \(S\). Therefore the conditional
expected number of successful side offers generated at step \(r+1\) is
at most
\[
        \frac1{M_r}
        \sum_{S\in\mathcal A_r}
        \widehat I_r(|S|-1)\frac1{|\mathcal A_r|}
        \le \frac{3b_n}{n},
\]
for \(r<r_n^{\max}\wedge\sigma\), since
\(M_r\ge n^2/3\), \(\widehat I_r\le n\), and
\(|S|-1\le b_n\).

Let \(H_n^{\mathrm{st}}\) be the total number of successful side offers
before \(r_n^{\max}\wedge\sigma\). Since
\(r_n^{\max}=O(A_n n)\), \(\mathbb E H_n^{\mathrm{st}}=O(A_n b_n)\).
Every component in \(\mathcal P_n^{\mathrm{pre}}\) is generated by at
least one such successful side offer, and on the event under
consideration every such component has at most \(b_n\) vertices.
Consequently,
\[
        \sum_{S\in\mathcal P_n^{\mathrm{pre}}}|S|^q
        \le b_n^q H_n^{\mathrm{st}}.
\]
Thus
\[
        \mathbb E\!\left[
        \sum_{S\in\mathcal P_n^{\mathrm{pre}}}|S|^q;
        \ \tau_1\le r_n^{\max},\ \sigma>r_n^{\max}
        \right]
        =O(A_n b_n^{q+1})
        =o(n/b_n^M).
\]
Markov's inequality, together with
\[
        \mathbb P\bigl(
        \tau_1>r_n^{\max}\ \text{or}\ \sigma\le r_n^{\max}
        \bigr)=o(1),
\]
proves the claim.
\end{proof}

We next construct a linear-sized family of persistent clean universal vertices. For convenience, let $\mathsf{Reg}_n$ be the event that all of the following conditions hold: $
        \tau_1\le r_n^{\max},
        I_{\tau_1}(G)\le\varepsilon_n n
        \widehat R_{\tau_1}\le b_n,
        S(\widehat G_{\tau_1})\le50n, $ 
and at least $\rho_1n/2$ unmarked $K_2$ are present at time $\tau_1$. As discussed before, we have $\mathbb P(\mathsf{Reg}_n)\to1$.

On \(\mathsf{Reg}_n\), choose \(\lfloor\rho_1n/2\rfloor\) unmarked copies of \(K_2\) present at time \(\tau_1\) by a fixed deterministic rule.  Let \(\mathcal K_0^{\mathrm{live}}\) be this family, and for every \(K\in\mathcal K_0^{\mathrm{live}}\), choose one endpoint \(u(K)\) by a fixed deterministic rule.  We call \(K\) live until the first accepted I--N edge after \(\tau_1\) whose non-isolated endpoint lies in the component \(K\); at that time \(K\) is retired.  Since at most \(\varepsilon_n n\) such retirements can occur, for all sufficiently large \(n\) at least \(\rho_1n/3\) selected copies remain live throughout the post-\(\tau_1\) evolution.

Let \(T_0=\tau_1\), and, for \(j\ge0\), let \(T_{j+1}\) be the
next time after \(T_j\) at which an I--I or I--N edge is accepted.
If no such time exists, set \(T_{j+1}=\infty\) and \(T_k=\infty\)
for all \(k>j+1\). Events asserting a type of accepted edge at $T_{j+1}$ are then interpreted as empty. An isolated vertex at time $T_j$ has remained isolated throughout the interval $[\tau_1,T_j]$. For every live selected copy \(K\), the chosen endpoint \(u(K)\) has been universal since \(K\) was born. Hence an edge from such an isolated vertex to \(u(K)\) is unoffered at time \(T_j\), and Lemma~\ref{lem:persistent-centres} keeps it acceptable while \(K\) remains live.

\begin{lemma}\label{lem:posttau-race-estimates}
On $\mathsf{Reg}_n$, conditional on $\mathcal F_{T_j}$, the following hold for every $j$ with $T_j<\infty$. If $D$ is a current non-trivial component, then
\begin{equation}\label{eq:race-attach}
\mathbb P\bigl(T_{j+1}\text{ attaches an isolated vertex to }D
       \mid\mathcal F_{T_j}\bigr)
\le \frac{4|D|}{\rho_1n}.
\end{equation}
Moreover,
\begin{equation}\label{eq:race-ii}
\mathbb P\bigl(T_{j+1}\text{ is an I--I edge}
       \mid\mathcal F_{T_j}\bigr)
\le \frac{3\varepsilon_n}{2\rho_1}.
\end{equation}
\end{lemma}

\begin{proof}
Fix $j$ and condition on $\mathcal F_{T_j}$.  Let $I_j$ be the set of vertices
isolated at time $T_j$, and let \(\mathcal K_j^{\mathrm{live}}\) be the family of selected copies of
\(K_2\) that are still live.  For a target component \(D\), put
\[
\mathcal K_j(D) = \{K\in\mathcal K_j^{\mathrm{live}}: V(K)\cap V(D)=\varnothing\}.
\]
By Proposition~\ref{pro:structural-invariants}(i), a current non-trivial component \(D\) can intersect at most one of the selected copies of \(K_2\). Hence, for all sufficiently large \(n\), \( |\mathcal K_j(D)|\ge |\mathcal K_j^{\mathrm{live}}|-1\ge \rho_1n/4\).
Let $R_j$ be the set of
unrevealed pairs at time $T_j$, and define
\[
        A_D=\{vx:v\in I_j,\ x\in V(D)\}\cap R_j,
        \qquad
        B_D=\{v u(K):v\in I_j,\ K\in\mathcal K_j(D)\}.
\]
The sets $A_D$ and $B_D$ are disjoint.  Every edge in $B_D$ is unoffered at time $T_j$ and would be
acceptable before the next accepted I--I or I--N edge. Thus, the event in
\eqref{eq:race-attach} is contained in the event that the first pair of
$A_D\cup B_D$ in the remaining suffix lies in $A_D$.  If $I_j=\varnothing$,
the probability is zero; otherwise Lemma~\ref{lem:suffix-exchangeability}
gives
\[
        \frac{|A_D|}{|A_D|+|B_D|}
        \le\frac{4|I_j||D|}{|I_j|\rho_1n}
        =\frac{4|D|}{\rho_1 n}.
\]

For \eqref{eq:race-ii}, use instead
\[
        A_{II}=\binom{I_j}{2}\cap R_j,
        \qquad
        B_{II}=\{v u(K):v\in I_j,\ K\in\mathcal K_j^{\mathrm{live}}\}.
\]
The same argument gives, when $I_j\ne\varnothing$,
\[
        \frac{|A_{II}|}{|A_{II}|+|B_{II}|}
        \le\frac{3\binom{|I_j|}{2}}{|I_j|\rho_1n}
        \le\frac{3|I_j|}{2\rho_1n}
        \le\frac{3\varepsilon_n}{2\rho_1 }.
\]
\end{proof}

A useful consequence of \eqref{eq:race-attach} is that a non-trivial component of current size $x$ has conditional expected $q$-moment increase at the next relevant step at most
\begin{equation}\label{eq:one-step-growth}
        \frac{4x}{\rho_1n}\bigl((x+1)^q-x^q\bigr)
        \le\frac{\gamma_qx^q}{n}.
\end{equation}

We now control the components created or first touched after $\tau_1$.

\begin{lemma}
\label{lem:posttau-activated-descendants}
Let \(\mathcal D_n^{\mathrm{act}}\) be the family of terminal components \(D\) satisfying
at least one of the following conditions:
\begin{enumerate}[label=(\roman*)]
\item \(D\) contains an unmarked non-trivial component present at time \(\tau_1\) whose component receives an accepted I--N edge after \(\tau_1\);
\item \(D\) contains both endpoints of an accepted I--I edge offered after
\(\tau_1\).
\end{enumerate}
Then for every fixed $q\ge1$ and $M>0$,
\[
        \sum_{D\in\mathcal D_n^{\mathrm{act}}}|D|^q
        =o_{\mathbb P}(n/b_n^M).
\]
\end{lemma}

\begin{proof}
We work on $\mathsf{Reg}_n$. View the isolated vertices at $\tau_1$ as singleton components. Let $Z_q(j)$ be the sum of the $q$-th powers of the sizes of all activated components at time $T_j$, and put
\(
        \overline Z_q(j)=\mathbf1_{\mathsf{Reg}_n}Z_q(j)
\).
For an unmarked non-trivial component $S$, if $S$ has not yet been touched by an isolated vertex, then by \eqref{eq:race-attach} its first touch at $T_{j+1}$ has conditional probability at most $4|S|/(\rho_1 n)$ and contributes at most $(|S|+1)^q$. Hence, on $\mathsf{Reg}_n$, the total conditional first-touch contribution is at most
\[
        \frac4{\rho_1n}\sum_{\substack{S\in
        \widehat{\mathcal P}_{\tau_1}\\ |S|\ge2,S\cap \widehat B_{\tau_1}=\emptyset }}
        |S|(|S|+1)^q
        \le \gamma_qb_n^{q-1},
\]
where we used $|S|\le b_n$ and $\sum_S |S|^2\le50n$. By \eqref{eq:race-ii}, the conditional probability of a post-$\tau_1$ I--I birth is at most $\Theta(\varepsilon_n)$, and such a birth contributes at most $2^q$.  Finally, by \eqref{eq:one-step-growth}, summed over all $D\in \mathcal D_n^{\mathrm{act}}$, bounds their subsequent growth by $\gamma_qZ_q(j)/n$.
Therefore,
\[
 \mathbb E[\Delta\overline Z_q(j)\mid\mathcal F_{T_j}]
 \le
 \bigl(\gamma_qb_n^{q-1}+\gamma_q\varepsilon_n\bigr)
       \mathbf1_{\mathsf{Reg}_n}
 +\frac{\gamma_q}{n}\overline Z_q(j).
\]
There are at most $\varepsilon_n n$ relevant accepted I--I or I--N steps after
$\tau_1$.  Lemma~\ref{lem:discrete-gronwall} gives
\[
 \mathbb E\bigl[\overline Z_q(\mathrm{terminal})
       \mid\mathcal F_{\tau_1}\bigr]
 \le
 \gamma_q\bigl(\varepsilon_n n b_n^{q-1}+\varepsilon_n^2n\bigr)
 e^{\gamma_q\varepsilon_n}\mathbf1_{\mathsf{Reg}_n}.
\]
By \eqref{eq:epsilon-envelope}, the right-hand side is
$o(n/b_n^M)$.  Markov's inequality and $\mathbb P (\mathsf{Reg}_n)\to 1$ complete the proof.
\end{proof}

The marked components are already exceptional at time $\tau_1$.

\begin{lemma}
\label{lem:terminal-marked-descendants}
Let $\mathcal M_n$ be the family of terminal components that contain a component of
$G_{\tau_1}$ contained in $\widehat B_{\tau_1}$ and are not already counted in
$\mathcal D_n^{\mathrm{act}}$.  For every fixed $q\ge1$ and $M>0$,
\[
        \sum_{D\in\mathcal M_n}|D|^q
        =o_{\mathbb P}(n/b_n^M).
\]
\end{lemma}

\begin{proof}
On $\mathsf{Reg}_n$, let $W_q(j)$ be the sum of the $q$-th powers of the sizes
of the terminal components that contain a component of
$G_{\tau_1}$ still not counted by the activated process at time
$T_j$, and set
\(
        \overline W_q(j)=\mathbf1_{\mathsf{Reg}_n}W_q(j).
\)
The only way a terminal component that contains a component of
$G_{\tau_1}$ can leave \(W_q\) is through an I--I birth involving a marked isolated vertex or through the first touch of an unmarked non-trivial component; these are precisely the two source types covered by Lemma~\ref{lem:posttau-activated-descendants}. Otherwise it remains in \(W_q\), and its subsequent I--N growth is controlled by \eqref{eq:one-step-growth}.
Thus $W_q$ has no positive birth or first-touch source term.  By
\eqref{eq:one-step-growth},
\[
        \mathbb E[\Delta\overline W_q(j)\mid\mathcal F_{T_j}]
        \le\frac{\gamma_q}{n}\overline W_q(j).
\]
Since there are at most $\varepsilon_n n$ relevant steps,
Lemma~\ref{lem:discrete-gronwall} yields
\[
        \mathbb E\bigl[\overline W_q(\mathrm{term})
              \mid\mathcal F_{\tau_1}\bigr]
        \le e^{\gamma_q\varepsilon_n}\overline W_q(0),
        \qquad
        \overline W_q(0)
        \le\mathbf1_{\mathsf{Reg}_n}R_{q,\tau_1}.
\]
By Lemma~\ref{lem:marked-region-control},
$R_{q,\tau_1}=o_{\mathbb P}(n/b_n^M)$.  Markov's inequality conditionally on
$\mathcal F_{\tau_1}$, followed by $\mathbb P (\mathsf{Reg}_n)\to 1$, proves the claim.
\end{proof}

Recall from Theorem~\ref{thm:intro-terminal-decomposition} that \(\mathcal E_n\) denotes the family of terminal components not accounted for by the component representation. We now make this definition precise in the preceding estimates, which can now be assembled in a single argument.

\begin{lemma}\label{lem:exceptional-components}
Let \(\mathcal E_n\) be the set of distinct terminal components belonging to at
least one of the following families:
\begin{enumerate}[label=(\alph*),leftmargin=*]
\item terminal component contains a time-$\tau_1$ component meeting $\widehat B_{\tau_1}$;
\item terminal component contains a size-exceptional star present at $\tau_1$;
\item terminal component contains a pre-exposed star present at $\tau_1$;
\item terminal component contains a post-$\tau_1$-touched clean star;
\item terminal component contains a post-$\tau_1$-born component;
\item terminal isolated vertices.
\end{enumerate}
Every terminal component that is not the terminal descendant of an
untouched \(\tau_1\)-clean star belongs to \(\mathcal E_n\).  Moreover, for every fixed
$q\ge1$ and $M>0$,
\[
        \sum_{D\in\mathcal E_n}|D|^q
        =o_{\mathbb P}(n/b_n^M).
\]
In particular, the total number of vertices in a component in $\mathcal E_n$ is
$o_{\mathbb P}(n)$.
\end{lemma}

\begin{proof}
The coverage assertion follows directly from the definitions.  A terminal
non-isolated component either has a non-trivial ancestor at time $\tau_1$ or is
born from a post-$\tau_1$ accepted I--I edge.  In the first case, its ancestor
is marked, size-exceptional, pre-exposed, or clean.  A clean ancestor is either
post-$\tau_1$-touched or untouched.  Terminal isolated vertices form the final
listed class.

It remains to sum the estimates. Components containing marked vertices
are controlled by Lemmas~\ref{lem:posttau-activated-descendants} and
\ref{lem:terminal-marked-descendants}. A pre-exposed star that later
receives an accepted I--N attachment, every post-\(\tau_1\)-touched clean
star, and every post-\(\tau_1\)-born component are counted in
\(\mathcal D_n^{\mathrm{act}}\). If a pre-exposed star never receives a
later accepted I--N attachment, then its terminal descendant has the same
vertex set, so its \(q\)-moment contribution is controlled directly by
Lemma~\ref{lem:preexposed-stars}. By
Lemma~\ref{lem:auxiliary-component-bounds}, w.h.p. no skeleton component
present before \(r_n^{\max}\) has more than \(b_n\) vertices, so no star
at \(\tau_1\) is size-exceptional. Finally, the number of terminal
isolated vertices is at most
\(
        I_{\tau_1}(G)\le\varepsilon_n n=o(n/b_n^M)
\). Summing the preceding bounds, allowing terminal components to be counted more than once, gives $o_{\mathbb P}(n/b_n^M)$.
\end{proof}

\subsection{Terminal component and statistic limits}

Let \(\mathcal S_n\) be the family of all \(\tau_1\)-clean stars, and let
\(\mathcal S_n^{\mathrm{unt}}\subseteq\mathcal S_n\) be the subfamily of
stars that are not post-\(\tau_1\)-touched. For \(S\in\mathcal S_n\) with at least two leaves, write \(c(S)\) for its unique center.  If \(S\cong K_2\), choose \(c(S)\) by a fixed deterministic rule.  In either case, let \(L(S):=V(S)\setminus\{c(S)\}\), and write
\(
\ell(S):=|L(S)|
\).

We first show that the evolution on the leaf sets of the clean stars can
be viewed as a family of independent finite induced-\(P_4\)-free
processes.

For each \(S\in\mathcal S_n\), let \(\pi_S\) be the relative order
induced by the remaining edge order on
\(
E_S^{\mathrm{leaf}}:=\binom{L(S)}2,
\)
and let \(\Gamma_S\) be the terminal graph of the
induced-\(P_4\)-free process on \(L(S)\) driven by \(\pi_S\).
Let \(\widetilde D_S\) be the graph on \(V(S)\) obtained by joining
\(c(S)\) to every vertex of \(\Gamma_S\).
\begin{lemma}
\label{lem:clean-star-completion}

Conditional on \(\mathcal F_{\tau_1}\), the graphs
\(
\bigl(\widetilde D_S\bigr)_{S\in\mathcal S_n}
\)
are mutually independent.  Moreover, if \(S\) has \(\ell\) leaves, then
\(
\Gamma_S\stackrel{d}{=}\mathcal G_\ell
\) and
\(
\widetilde D_S\stackrel{d}{=}
\mathcal H_\ell=K_1\vee\mathcal G_\ell.
\)

For every \(S\in\mathcal S_n^{\mathrm{unt}}\), the actual terminal
component of \(G_N\) containing \(S\) is exactly
\(\widetilde D_S\), as a labelled graph.
\end{lemma}

\begin{proof}
By property (P3) in Definition~\ref{def:tau1-clean}, no pair in
\(E_S^{\mathrm{leaf}}\) has been offered by time \(\tau_1\).  Conditional
on \(\mathcal F_{\tau_1}\), the sets
\(
\bigl(E_S^{\mathrm{leaf}}\bigr)_{S\in\mathcal S_n}
\)
are fixed and pairwise disjoint.  Lemma~\ref{lem:suffix-exchangeability}
therefore implies that the relative orders \((\pi_S)_{S\in\mathcal S_n}\)
are mutually independent uniform orders.  Consequently, the graphs
\((\Gamma_S)_{S\in\mathcal S_n}\), and hence the graphs
\((\widetilde D_S)_{S\in\mathcal S_n}\), are conditionally independent.
As \(\ell(S)=\ell\), their definitions give
\(
\Gamma_S\stackrel{d}{=}\mathcal G_\ell
\) and \(
\widetilde D_S\stackrel{d}{=}
K_1\vee\mathcal G_\ell=\mathcal H_\ell
\).

Now fix \(S\in\mathcal S_n^{\mathrm{unt}}\).  Since \(S\) is untouched,
no accepted I--N edge after \(\tau_1\) adds a vertex to the component
containing \(S\).  Moreover, distinct non-trivial components cannot merge,
by Proposition~\ref{pro:structural-invariants}.  Thus the vertex set of
this component remains \(V(S)\), and its only possible accepted edges
after \(\tau_1\) have both endpoints in \(L(S)\).

The vertex \(c(S)\) remains universal throughout this evolution.  For
every graph \(H\) on \(L(S)\),
\(K_1\vee H\) is induced-\(P_4\)-free if and only if
$H$ is induced-\(P_4\)-free,
because a universal vertex cannot belong to an induced copy of \(P_4\).
It follows that each offered leaf--leaf edge is accepted in the original
process precisely when it is accepted by the induced-\(P_4\)-free
process on \(L(S)\) driven by \(\pi_S\).  Hence the actual terminal
component containing \(S\) is \(\widetilde D_S\).
\end{proof}

The exceptional-component estimate implies that deleting the stars that are touched or unclean does not change either the fixed-size asymptotics or the polynomial-moment tails of the skeleton stars. For \(\ell\ge1\), define
\[
Y_\ell^{\mathrm{cl}}
:=\bigl|{S\in\mathcal S_n:\ell(S)=\ell}\bigr|,
\qquad
Y_\ell^{\mathrm{unt}}
:=\bigl|{S\in\mathcal S_n^{\mathrm{unt}}:
\ell(S)=\ell}\bigr|.
\]

\begin{cor}
\label{cor:clean-star-counts}
For every fixed \(\ell\ge1\),
\[
Y_\ell^{\mathrm{unt}}
=\rho_\ell n+o_{\mathbb P}(n),
\qquad
Y_\ell^{\mathrm{cl}}-Y_\ell^{\mathrm{unt}}
=o_{\mathbb P}(n).
\]

Moreover, for every fixed \(q\ge1\) and \(\eta>0\),
\[
\lim_{L_0\to\infty}\limsup_{n\to\infty}
\mathbb P\!\left(
\frac1n
\sum_{\substack{S\in\mathcal S_n\\ |S|>L_0}}
|S|^q>\eta
\right)=0.
\]
The same conclusion therefore holds with \(\mathcal S_n^{\mathrm{unt}}\) in place of \(\mathcal S_n\).
\end{cor}

\begin{proof}
Let \(\widehat Y_{\ell,\tau_1}\) be the number of I-skeleton
components with \(\ell\) leaves at time \(\tau_1\).  By
Theorem~\ref{thm:skeleton-comparison},
\(
\widehat I_{\tau_1}
\le I_{\tau_1}(G)+o_{\mathbb P}(n)
=o_{\mathbb P}(n).
\)
While every accepted I-skeleton edge after \(\tau_1\) consumes at least one
vertex that is isolated at time \(\tau_1\), for every fixed \(\ell\),
\[
\widehat Y_{\ell,\tau_1}
=\widehat Y_\ell+o_{\mathbb P}(n)
=\rho_\ell n+o_{\mathbb P}(n),
\]
where the second equality follows from
Theorem~\ref{thm:intro-iskeleton}.

Let \(\widehat{\mathcal S}_{\tau_1}:=\{C\in\widehat{\mathcal P}_{\tau_1}:|C|\ge2\}\)
be the family of all non-trivial I-skeleton stars present at
\(\tau_1\), without imposing cleanliness. Every star in
\(\widehat{\mathcal S}_{\tau_1}\setminus \mathcal S_n^{\mathrm{unt}}\) is either marked, size-exceptional, pre-exposed, or a touched clean star. Hence all of its vertices belong to terminal
components in \(\mathcal E_n\). Since the stars at time \(\tau_1\) are
vertex-disjoint,
\[
\sum_{S\in
\widehat{\mathcal S}_{\tau_1}\setminus\mathcal S_n^{\mathrm{unt}}}|S|
\le
\sum_{D\in\mathcal E_n}|D|
=o_{\mathbb P}(n).
\]
Thus deleting all stars outside \(\mathcal S_n^{\mathrm{unt}}\) from the
\(\widehat Y_{\ell,\tau_1}\) stars with \(\ell\) leaves proves
\(
Y_\ell^{\mathrm{unt}}
=\rho_\ell n+o_{\mathbb P}(n).
\)
Since
\(\mathcal S_n\setminus\mathcal S_n^{\mathrm{unt}}\) is a subfamily of
the deleted stars, we also have
\(
Y_\ell^{\mathrm{cl}}-Y_\ell^{\mathrm{unt}}
=o_{\mathbb P}(n).
\)

For the tail estimate, work on the event
\({\tau_1\le r_n^{\max}}\), whose probability tends to one.  Every
non-trivial I-skeleton component present at \(\tau_1\) is contained in a
unique component at time \(r_n^{\max}\), distinct components remain
distinct, and component sizes can only increase.  Therefore
\[
\sum_{\substack{S\in\mathcal S_n\ |S|>L_0}}|S|^q
\le
\sum_{\substack{
C\in\widehat{\mathcal P}_{r_n^{\max}}\
|C|>L_0}}
|C|^q.
\]
The required conclusion follows from
Lemma~\ref{lem:auxiliary-component-bounds}.  Restricting the sum to
\(\mathcal S_n^{\mathrm{unt}}\) can only decrease it.
\end{proof}

We now complete the proof of the terminal component decomposition.

\begin{proof}[\bfseries\upshape Proof of Theorem~\ref{thm:intro-terminal-decomposition}]
For every \(S\in\mathcal S_n\), let \(\widetilde D_S\) be the completion
defined in Lemma~\ref{lem:clean-star-completion}.  Conditional on
\(\mathcal F_{\tau_1}\), these completions are mutually independent, and
\(
\widetilde D_S
\stackrel{d}{=}\mathcal H_{\ell(S)}.
\)
For every \(S\in\mathcal S_n^{\mathrm{unt}}\), the same lemma identifies
\(\widetilde D_S\) with the actual terminal component containing \(S\).
Consequently,
\[
G_N\!\left[
\bigcup_{S\in\mathcal S_n^{\mathrm{unt}}}V(S)
\right]
=
\bigsqcup_{S\in\mathcal S_n^{\mathrm{unt}}}
\widetilde D_S
\]
as labelled graphs. Thus the terminal components arising from untouched clean stars are
coupled to the corresponding members of a single conditionally
independent family of finite completions.

Every terminal component not represented in the preceding union belongs
to \(\mathcal E_n\).  By
Lemma~\ref{lem:exceptional-components}, for every fixed \(q\ge1\),
\[
\sum_{D\in\mathcal E_n}|D|^q
=o_{\mathbb P}(n).
\]
In particular, the exceptional components contain
\(o_{\mathbb P}(n)\) vertices.

Finally, Corollary~\ref{cor:clean-star-counts} gives, for every fixed
\(\ell\ge1\),
\(
\bigl|{S\in\mathcal S_n^{\mathrm{unt}}:
\ell(S)=\ell}\bigr|
=\rho_\ell n+o_{\mathbb P}(n).
\)
Each corresponding terminal component has distribution
\(\mathcal H_\ell\).  This proves both the asserted component
representation and the stated component densities.
\end{proof}

\begin{proof}[\bfseries\upshape Proof of Corollary~\ref{cor:fixed-component-limits}]
Fix a connected graph \(F\) with \(|F|\ge2\), and put
\(\ell=|F|-1\). Let
\[
        \mathcal S_{\ell,n}^{\mathrm{cl}}
        :=\{S\in\mathcal S_n:\ell(S)=\ell\},
        \qquad
        p_F:=\mathbb P(\mathcal H_\ell\cong F).
\]
For \(S\in\mathcal S_{\ell,n}^{\mathrm{cl}}\), define
\[
        \widetilde X_S
        :=\mathbf1_{\{\widetilde D_S\cong F\}},
\]
where \(\widetilde D_S\) is the completion from
Lemma~\ref{lem:clean-star-completion}. Conditional on
\(\mathcal F_{\tau_1}\), the variables
\((\widetilde X_S)_{S\in\mathcal S_{\ell,n}^{\mathrm{cl}}}\) are
independent Bernoulli variables with success probability \(p_F\).
Therefore
\[
\operatorname{Var}\!\left(
        \sum_{S\in\mathcal S_{\ell,n}^{\mathrm{cl}}}\widetilde X_S
        \,\middle|\,\mathcal F_{\tau_1}
\right)
\le Y_\ell^{\mathrm{cl}}\le n.
\]
By conditional Chebyshev's inequality and
Corollary~\ref{cor:clean-star-counts},
\[
\frac1n\sum_{S\in\mathcal S_{\ell,n}^{\mathrm{cl}}}\widetilde X_S-\frac{Y_\ell^{\mathrm{cl}}}{n}p_F
\xrightarrow{\mathbb P}0,
\qquad
\frac{Y_\ell^{\mathrm{cl}}}{n}
\xrightarrow{\mathbb P}\rho_\ell.
\]
Hence
\[
\frac1n\sum_{S\in\mathcal S_{\ell,n}^{\mathrm{cl}}}\widetilde X_S
\xrightarrow{\mathbb P}\rho_\ell p_F.
\]

For untouched clean stars the artificial completion is the actual
terminal component. The discrepancy between the preceding full clean
completion count and \(C_F(G_N)\) is therefore bounded by the number of
touched clean stars plus the number of exceptional terminal components:
\[
\left|
C_F(G_N)-\sum_{S\in\mathcal S_{\ell,n}^{\mathrm{cl}}}\widetilde X_S
\right|\le
|\mathcal S_n\setminus\mathcal S_n^{\mathrm{unt}}|
+|\mathcal E_n|
\le
2\sum_{D\in\mathcal E_n}|D|
=o_{\mathbb P}(n).
\]
Dividing by \(n\) proves
\[
\frac{C_F(G_N)}n
\xrightarrow{\mathbb P}
\rho_{|F|-1}\mathbb P(\mathcal H_{|F|-1}\cong F).
\]
\end{proof}

We next give general additive statistics from the component representation.

\begin{proof}[\bfseries\upshape Proof of Theorem~\ref{thm:additive-statistics}]
For \(S\in\mathcal S_n\), let \(\widetilde D_S\) be the completion from
Lemma~\ref{lem:clean-star-completion}.  We first compare the actual
terminal graph with the full independent completion family.

Choose an integer \(q\ge1\) and a constant \(C<\infty\) such that
\(|\phi(F)|\le C|V(F)|^q\) for every finite connected graph \(F\). The
components of \(G_N\) outside the family indexed by
\(\mathcal S_n^{\mathrm{unt}}\) belong to \(\mathcal E_n\). Moreover,
every star in
\(\mathcal S_n\setminus\mathcal S_n^{\mathrm{unt}}\) has all its vertices
in the exceptional region. Hence
\[
\sum_{S\in\mathcal S_n\setminus\mathcal S_n^{\mathrm{unt}}}|S|^q
\le
\sum_{D\in\mathcal E_n}|D|^q
=o_{\mathbb P}(n).
\]
Since \(|V(\widetilde D_S)|=|S|\),
\[
\left|
\sum_{D\in\mathcal C(G_N)}\phi(D)
-\sum_{S\in\mathcal S_n}\phi(\widetilde D_S)
\right|
\le
C\sum_{D\in\mathcal E_n}|D|^q
+
C\sum_{S\in\mathcal S_n\setminus\mathcal S_n^{\mathrm{unt}}}|S|^q
=o_{\mathbb P}(n).
\]
In particular,
\begin{equation}\label{eq:phi-reduction}
\sum_{D\in\mathcal C(G_N)}\phi(D)
=
\sum_{S\in\mathcal S_n}\phi(\widetilde D_S)
+o_{\mathbb P}(n).
\end{equation}

Fix \(L_0<\infty\).  Conditional on \(\mathcal F_{\tau_1}\), the random
variables
\(
\bigl(
\phi(\widetilde D_S)\bigr)_{
S\in\mathcal S_n,\ \ell(S)\le L_0
}
\)
are independent and uniformly bounded.  Their total conditional variance
is \(O(n)\).  Hence Chebyshev's inequality gives
\[
\frac1n
\sum_{\substack{S\in\mathcal S_n\ \ell(S)\le L_0}}
\left(
\phi(\widetilde D_S)
-\mathbb E\!\left[
\phi(\widetilde D_S)\mid\mathcal F_{\tau_1}
\right]
\right)
\xrightarrow{\mathbb P}0.
\]
If \(\ell(S)=\ell\), then
\(
\mathbb E\!\left[
\phi(\widetilde D_S)\mid\mathcal F_{\tau_1}
\right]
=\mathbb E\phi(\mathcal H_\ell).
\)
Corollary~\ref{cor:clean-star-counts} therefore yields
\[
\frac1n
\sum_{\substack{S\in\mathcal S_n\ \ell(S)\le L_0}}
\phi(\widetilde D_S)
\xrightarrow{\mathbb P}
\sum_{\ell\le L_0}
\rho_\ell\mathbb E\phi(\mathcal H_\ell).
\]

The remaining random terms satisfy
\[
\frac1n
\sum_{\substack{S\in\mathcal S_n\ \ell(S)>L_0}}
|\phi(\widetilde D_S)|
\le
\frac{C}{n}
\sum_{\substack{S\in\mathcal S_n\ |S|>L_0+1}}
|S|^q.
\]
By Corollary~\ref{cor:clean-star-counts}, the right-hand side is uniformly
negligible as \(L_0\to\infty\).

Finally,
\(
|\mathbb E\phi(\mathcal H_\ell)|
\le C(\ell+1)^q
\), and Corollary~\ref{cor:iskeleton-factorial-tail} gives a factorial tail for \((\rho_\ell)_{\ell\ge1}\).  Thus
\[
\sum_{\ell\ge1}
\rho_\ell|\mathbb E\phi(\mathcal H_\ell)|
<\infty.
\]
Letting \(L_0\to\infty\) in \eqref{eq:phi-reduction} proves
\[
\frac1n
\sum_{D\in\mathcal C(G_N)}\phi(D)
\xrightarrow{\mathbb P}
\sum_{\ell\ge1}
\rho_\ell\mathbb E\phi(\mathcal H_\ell).
\]
\end{proof}

\begin{proof}[\bfseries\upshape Proofs of Corollaries~\ref{cor:intro-clique-profile} and~\ref{cor:intro-degree-distribution}]
For clique counts, take
\(
\phi(D)=N(D,K_j)
\)
for fixed \(j\ge2\).  Then
\[
\frac{N(G_N,K_j)}n
\xrightarrow{\mathbb P}
\kappa_j
=
\sum_{\ell\ge1}\rho_\ell
\mathbb E N(\mathcal H_\ell,K_j).
\]
In
\(\mathcal H_\ell=K_1\vee\mathcal G_\ell\), a clique either avoids the
universal vertex or contains it.  Therefore
\[
N(\mathcal H_\ell,K_j)
=
N(\mathcal G_\ell,K_j)
+
N(\mathcal G_\ell,K_{j-1}).
\]
For \(j=2\), this becomes
\(
e(\mathcal H_\ell)
=
\ell+e(\mathcal G_\ell),
\)
which gives the stated formula for \(\kappa_2\).

For degree counts, take
\(
\phi(D)=N_k^{\deg}(D).
\)
This gives the limit defining \(p_k\).  In
\(\mathcal H_\ell=K_1\vee\mathcal G_\ell\), the universal vertex has
degree \(\ell\), while the degree of every leaf-set vertex is increased
by one.  Hence, for \(k\ge1\),
\[
N_k^{\deg}(\mathcal H_\ell)
=
\mathbf1_{{\ell=k}}
+
N_{k-1}^{\deg}(\mathcal G_\ell),
\]
whereas
\(
N_0^{\deg}(\mathcal H_\ell)=0.
\)
It follows that \(p_0=0\). Applying Theorem~\ref{thm:additive-statistics} to
\(\phi(D)=|V(D)|\) gives
\[1=
\frac1n\sum_{D\in\mathcal C(G_N)}|V(D)|
\ \xrightarrow{\mathbb P}\
\sum_{\ell\ge1}(\ell+1)\rho_\ell,
\]
so
\(\sum_{\ell\ge1}(\ell+1)\rho_\ell=1\). Since all summands are
non-negative, Tonelli's theorem now gives
\[
\sum_{k\ge0}p_k
=
\sum_{\ell\ge1}\rho_\ell\,
\mathbb E\!\left[\sum_{k\ge0}N_k^{\deg}(\mathcal H_\ell)\right]
=
\sum_{\ell\ge1}(\ell+1)\rho_\ell
=1.
\]
\end{proof}
\subsection{Effective computation of the finite completion laws}
\label{subsec:effective-computation}

Lemma~\ref{lem:clean-star-completion} identifies the completion of an
untouched clean star with \(\ell\) leaves as
\(\mathcal H_\ell=K_1\vee\mathcal G_\ell\). The laws of the finite random graphs
\(\mathcal G_\ell\), and hence the constants appearing in the preceding
limit theorems, can be computed recursively.

Fix \(\ell\ge1\) and let $[\ell]=\{1,2,\cdots,\ell\}$ be the vertex set and \(E_\ell=\binom{[\ell]}2\). A state is a pair
\((A,U)\), where \(A\subseteq E_\ell\) is the set of accepted edges and
\(U\subseteq E_\ell\setminus A\) is the set of pairs not yet offered. For
\(e\in U\), let

$$
A_e=\begin{cases}
A\cup\{e\},& \text{if }([\ell],A\cup\{e\})\text{ is induced-\(P_4\)-free},\\
A,&\text{otherwise}.
\end{cases}
$$

For a graph functional \(\psi\), define
\(
g_{\ell,\psi}(A,\varnothing)=\psi(([\ell],A))
\)
and for \(U\ne\varnothing\),
\[
g_{\ell,\psi}(A,U)
=\frac1{|U|}\sum_{e\in U}
g_{\ell,\psi}(A_e,U\setminus\{e\}).
\]
Then
\(
\mathbb E\psi(\mathcal G_\ell)
=g_{\ell,\psi}(\varnothing,E_\ell).
\)

For example, writing
\[
\gamma_{\ell,j}:=\mathbb E N(\mathcal G_\ell,K_j),
\qquad \gamma_{\ell,1}:=\ell,
\]
gives
\[
\kappa_j
=\sum_{\ell\ge1}\rho_\ell
\bigl(\gamma_{\ell,j}+\gamma_{\ell,j-1}\bigr),
\qquad j\ge2,
\]
and
\[
0\le
\kappa_j-\sum_{\ell\le L}\rho_\ell
\bigl(\gamma_{\ell,j}+\gamma_{\ell,j-1}\bigr)
\le
\sum_{\ell>L}\rho_\ell\binom{\ell+1}{j}.
\]
More generally, if
\(|\phi(D)|\le q_1|V(D)|^{q_2}\) with \(q_1>0\) and \(q_2\ge0\), then
\[
\left|
\sum_{\ell\ge1}\rho_\ell\,\mathbb E\phi(\mathcal H_\ell)
-
\sum_{\ell\le L}\rho_\ell\,\mathbb E\phi(\mathcal H_\ell)
\right|
\le
q_1\sum_{\ell>L}(\ell+1)^{q_2}\rho_\ell.
\]
The factorial tail of \((\rho_\ell)\) therefore yields effective
approximations from finitely many completion laws.

\section{Acknowledgements}
	This work was supported by the Key Program of the National Natural Science Foundation of China (NSFC) under Grant No. 12231018.

\section{Declaration on the use of AI}
    During the preparation of this manuscript, ChatGPT 5.6 was used to assist with computational tasks and language polishing. All research ideas, proof strategies, theoretical arguments, and conclusions were independently developed by the authors. 

\bibliographystyle{abbrv}

\bibliography{references}
\end{document}